\documentclass[a4paper,reqno,11pt]{amsart}
\usepackage[T1]{fontenc}
\usepackage[utf8]{inputenc}
\usepackage[english]{babel}
\usepackage[dvipsnames,svgnames,table]{xcolor}
\usepackage{amssymb, amsmath, amsthm, graphicx, enumerate, color, mathtools, comment, caption, subcaption, float, lmodern}
\usepackage[foot]{amsaddr}
\usepackage[shortlabels]{enumitem}
\usepackage[unicode=true]{hyperref}
\usepackage[capitalise, noabbrev]{cleveref}
\usepackage{hyperref}
\usepackage{thmtools, thm-restate}
\usepackage{accents}
\usepackage[longnamesfirst,numbers,sort&compress]{natbib}

\hypersetup{ 
    colorlinks,
    linkcolor={RoyalBlue},
    citecolor={RubineRed},
    urlcolor={blue!80!black},
    pdftitle={A coarse Gallai theorem}
} 

\makeatletter
\def\thm@space@setup{
  \thm@preskip=4mm
  \thm@postskip=0mm
}
\makeatother

\usepackage[normalem]{ulem}

\usepackage{mdframed}

\mdfdefinestyle{dontsplit}{
  hidealllines=true,
  nobreak=true,
  leftmargin=0pt,
  rightmargin=0pt,
  innerleftmargin=0pt,
  innerrightmargin=0pt,
}

\crefformat{equation}{#2(#1)#3}
\let\eqref\cref
\crefformat{subsection}{Subsection #2#1#3}

\newmdtheoremenv[style=dontsplit]{theorem}{Theorem}
\newtheorem{lemma}[theorem]{Lemma}
\newmdtheoremenv[style=dontsplit]{obs}[theorem]{Observation}
\newmdtheoremenv[style=dontsplit]{remark}[theorem]{Remark}
\newmdtheoremenv[style=dontsplit]{proposition}[theorem]{Proposition}
\newmdtheoremenv[style=dontsplit]{question}[theorem]{Question} 
\newmdtheoremenv[style=dontsplit]{corollary}[theorem]{Corollary} 
\newmdtheoremenv[style=dontsplit]{problem}[theorem]{Problem}
\newmdtheoremenv[style=dontsplit]{conjecture}[theorem]{Conjecture}
\newtheorem*{theorem*}{Theorem}
\newtheorem*{corollary*}{Corollary} 
\newtheorem*{problem*}{Problem}
\newtheorem*{conjecture*}{Conjecture}
\newtheorem*{lemma*}{Lemma}
\newtheorem*{obs*}{Observation}
\newtheorem*{remark*}{Remark}
\newtheorem*{proposition*}{Proposition}
\newtheorem*{question*}{Question} 
\newtheorem*{example*}{Example}
\newtheorem*{claim*}{Claim} 

\theoremstyle{remark}
\newmdtheoremenv[style=dontsplit]{claim}[theorem]{Claim}
\crefname{claim}{Claim}{Claims}
\newmdtheoremenv[style=dontsplit]{example}[theorem]{Example}

\DeclarePairedDelimiter\set{\{}{\}}

\newcommand{\q}[1]{``#1''}
\newcommand{\defin}[1]{\emph{\textcolor{ForestGreen}{#1}}}

\newcommand{\Oh}{\mathcal{O}}

\newcommand{\calM}{\mathcal{M}} 
\newcommand{\calN}{\mathcal{N}}

\newcommand{\calP}{\mathcal{P}}

\newcommand{\calT}{\mathcal{T}}

\newcommand{\dist}{\mathrm{dist}}

\DeclareMathOperator\length{len}

\DeclareMathOperator\degree{deg}

\let\ge\geqslant
\let\leq\leqslant
\let\geq\geqslant

\let\subset\subseteq

\let\epsilon\varepsilon

\let\setminus\backslash

\graphicspath{{figs/}}

\setenumerate{label=\textup{(\roman*)}, noitemsep, topsep=0pt,
labelindent=.2em, leftmargin=*, widest=iii,}
\setitemize{noitemsep, topsep=-\parskip, labelindent=.2em, leftmargin=*, widest=iii,}

\newenvironment{enumerate'}{\begin{enumerate}[label=\textup{(\roman*')}, noitemsep, 
topsep=2pt plus 2pt, labelindent=.2em, leftmargin=*, widest=10]}{\end{enumerate}}

\newenvironment{enumerateAlpha}{\begin{enumerate}[label=\textup{(\alph*)}, noitemsep, 
topsep=2pt plus 2pt, labelindent=.2em, leftmargin=*, widest=iii]}{\end{enumerate}}

\newenvironment{enumerateAlpha'}{\begin{enumerate}[label=\textup{(\alph*$'$)}, noitemsep, 
topsep=2pt plus 2pt, labelindent=.2em, leftmargin=*, widest=iii]}{\end{enumerate}}

\newenvironment{enumerateAlpha''}{\begin{enumerate}[label=\textup{(\alph*$''$)}, noitemsep, 
topsep=2pt plus 2pt, labelindent=.2em, leftmargin=*, widest=iii]}{\end{enumerate}}

\makeatletter
\newcommand{\myitem}[1]{%
\item[#1]\protected@edef\@currentlabel{#1}%
}
\makeatother

\begin{document} 

\title[\MakeUppercase{A coarse Gallai theorem}] 
{\MakeUppercase{A coarse Gallai theorem}}

\author[Distel]{Marc Distel}
\address[Distel]{School of Mathematics, Monash University, Melbourne, Australia}
\email{\href{mailto:Marc.Distel@monash.edu}{marc.distel@monash.edu}}

\author[Giocanti]{Ugo Giocanti}
\email{\href{mailto:ugo.giocanti@uj.edu.pl}{ugo.giocanti@uj.edu.pl}}

\author[Hodor]{J\k{e}drzej Hodor}
\email{\href{mailto:jedrzej.hodor@gmail.com}{jedrzej.hodor@gmail.com}}

\author[Legrand-Duchesne]{Clément Legrand-Duchesne}
\email{\href{mailto:clement.legrand-duchesne@uj.edu.pl}{clement.legrand-duchesne@uj.edu.pl}}

\author[Micek]{Piotr Micek}
\email{\href{mailto:piotr.micek@uj.edu.pl}{piotr.micek@uj.edu.pl}}

\address[Giocanti, Legrand-Duchesne, Micek]{Theoretical Computer Science Department, Faculty of Mathematics and Computer Science, Jagiellonian University, Kraków, Poland.}
\address[Hodor]{Theoretical Computer Science Department, 
Faculty of Mathematics and Computer Science and  Doctoral School of Exact and Natural Sciences, Jagiellonian University, Krak\'ow, Poland}

\thanks{
The research was conducted during a two-week visit of Marc Distel at Jagiellonian University in November 2025.
M.\ Distel is supported by an Australian Government Research Training Program Scholarship.
U.\ Giocanti is supported by the National Science Center of Poland
under grant 2022/47/B/ST6/02837 within the OPUS 24 program.
J.\ Hodor is supported by a Polish Ministry of Education and Science grant (Perły Nauki; PN/01/0265/2022).
C.\ Legrand-Duchesne and P.\ Micek are supported by the National Science Center of Poland under grant
UMO-2023/05/Y/ST6/00079 within the WEAVE-UNISONO program.
}

\begin{abstract}
We prove that there exist functions $f$ and $g$ such that for all positive integers $k$ and $d$, 
for every graph $G$ and every subset $A$ of the vertices of $G$, 
either $G$ contains $k$ $A$-paths such that vertices of different $A$-paths are at distance at least $d$ in $G$, or 
there exists a set $X$ of the vertices of $G$ with $|X|\leq f(k)$ such that 
every $A$-path in $G$
contains a vertex of $B_G(X,g(k,d))$.
\end{abstract}

\maketitle

\section{Introduction}\label{sec:introduction}

In 1964, Gallai~\cite{Gallai64} proved that for every finite graph $G$ and every subset $A$ of the vertices of $G$, either $G$ contains $k$ vertex-disjoint $A$-paths, or there exists a subset $X$ of the vertices of $G$ with $|X|\leq 2k-2$ such that every $A$-path in $G$ contains a vertex of $X$.
Here and throughout, an \defin{$A$-path} in $G$ is a path with at least two vertices and both ends in $A$.

Inspired by Gromov's seminal work on coarse geometry \cite{Gromov}, Georgakopoulos and Papasoglou~\cite{GP23} initiated a systematic search for coarse metric analogs of fundamental statements in structural graph theory, expressing hopes that it could \q{evolve into a coherent theory that could be called \emph{coarse graph theory}}. 
Broadly speaking, coarse geometry (resp.\ graph theory) consists of studying structural properties of metric spaces (resp.\ graphs) from a large-scale perspective, usually by identifying points (resp.\ vertices) with balls of small radius. 
In this spirit, we prove the following coarse variant of Gallai's theorem.

\begin{theorem}
\label{theorem:main}
There exist functions $f:\mathbb N\to \mathbb N$ and $g:\mathbb N^2\to \mathbb N$ such that for all positive integers $k$ and $d$, for every graph $G$ and every subset $A$ of the vertices of $G$, either $G$ contains $k$ $A$-paths which are pairwise at distance at least $d$, or there exists a set $X$ of the vertices of $G$ with $|X|\leq f(k)$ such that every $A$-path in $G$ contains a vertex in $B_G(X,g(k,d))$.
\end{theorem}

Here, and throughout, \defin{$B_G(X,r)$} denotes the set of all vertices of $G$ at distance at most $r$ from $X$ in $G$.
Our proof of~\cref{theorem:main} gives $f(k) = 4k-4$ 
and $g(k,d) = d\cdot 256^k$.
Independently from our work, 
the statement of~\cref{theorem:main} was conjectured by Geelen in 2024.\footnote{Posed at the Barbados Graph Theory Workshop in March 2024 held at the Bellairs Research Institute of McGill University in Holetown.} 
The special case of $d=2$ and arbitrary $k$ was announced 
by Albrechtsen, Knappe, and Wollan~\cite{W24}, and later proved by Hickingbotham and Joret~\cite{HJ25} for functions $f(k)\in \mathcal{O}(k)$ and $g(k,2)=1$.

As mentioned, coarse graph theory roughly consists of studying graphs when identifying vertices with balls of small radius. 
The packing part of the original Gallai's theorem requires the paths to have distinct endpoints.
However, the packing part of~\Cref{theorem:main} allows paths to have both endpoints contained in a small ball. This motivates the following statement, which is better aligned with this requirement.
For a graph $G$, a subset $A$ of the vertices of $G$, and a positive integer $d$, 
an $A$-path is \defin{$d$-coarse} if the distance in $G$ between its endpoints is at least~$d$.
\begin{theorem}
\label{theorem:main2}
There exist functions $f:\mathbb N\to \mathbb N$ and $g:\mathbb N^2\to \mathbb N$ such that for all positive integers $k$ and $d$, 
for every graph $G$ and every subset $A$ of the vertices of $G$, 
either $G$ contains $k$ $d$-coarse $A$-paths which are pairwise at distance at least $d$, or 
there exists a set $X$ of the vertices of $G$ with $|X|\leq f(k)$ such that 
every $g(k,d)$-coarse $A$-path in $G$
contains a vertex in $B_G(X,g(k,d))$.
\end{theorem}
Note that the functions $f$ and $g$ for which we prove~\Cref{theorem:main2} below are the same as for~\Cref{theorem:main}. 

Gallai's theorem is similar in flavour to Menger's theorem~\cite{Men27}, published in 1927 and now regarded as one of the cornerstones of structural graph theory.
Menger's theorem states that for every graph $G$ and all subsets $S$ and $T$ of the vertices of $G$, either $G$ contains $k$ vertex-disjoint $S$-$T$ paths, 
or there exists a subset $X$ of the vertices of $G$ with $|X|\leq k-1$ such that every $S$-$T$ path in $G$ contains a vertex in $X$. 
Here, an \defin{$S$-$T$ path} in $G$ is a path between a vertex of $S$ and a vertex of $T$ in $G$. 
Finding a coarse analog of Menger's theorem was one of the first challenges in coarse graph theory. 
Several variants were conjectured, see for example, Georgakopoulos and Papasoglou \cite[Conjecture~1.4]{GP23}; Albrechtsen, Huynh, Jacobs, Knappe, and Wollan~\cite[Conjecture~3]{two-paths}; and Nguyen, Scott, and Seymour~\cite[Conjecture~4.2]{NGUYEN202568}.
The weakest considered version is as follows:
there are functions $f:\mathbb N^2\to \mathbb N$ and $g:\mathbb N^2\to \mathbb N$ such that 
for all positive integers $k$ and $d$, 
for every graph $G$ and all subsets $S$ and $T$ of the vertices of $G$, 
either $G$ contains $k$ $S$-$T$~paths which are pairwise at distance at least $d$, or there exists a subset $X$ of the vertices of $G$ with $|X|\leq f(k,d)$ such that every $S$-$T$ path in $G$ contains a vertex in $B_G(X,g(k,d))$. 
Unfortunately, this is now known to be false already for $k=d=3$ as proved by Nguyen, Scott, and Seymour~\cite{NGUYEN202568,NSS}. Interestingly, the case of $d=2$ (so-called induced case) is still wide open, while the case of $k=2$ is known to be true~\cite{GP23,two-paths,NGUYEN202568}.

Gallai's and Menger's theorems are among the most prominent examples of packing vs hitting statements in graph theory. 
Such results are often referred to as Erd\H{o}s--P\'{o}sa type theorems, after the celebrated theorem of Erd\H{o}s and P\'{o}sa~\cite{EP65} from 1965:
every graph $G$ either contains $k$ vertex-disjoint cycles or contains a set $X$ of $\Oh(k\log k)$ vertices such that every cycle in $G$ contains a vertex in $X$.
The following coarse version of this statement was proved by Dujmovi{\'c}, Joret, Micek, and Morin~\cite{DJMM25}: 
there exist functions $f,g:\mathbb{N}\to\mathbb{N}$ such that for all positive integers $k$ and $d$, for every graph $G$,  either $G$ contains $k$ cycles which are pairwise at distance at least $d$, or there exists a subset $X$ of the vertices of $G$ with $|X|\leq f(k)$ such that every cycle in $G$ contains a vertex in $B_G(X,g(d))$.

We conclude the introduction by proposing a conjecture that 
strengthens the statement of~\cref{theorem:main} in which 
the radius of the balls that hit all the $A$-paths 
is only a function of~$d$. 
\begin{conjecture}
\label{conj:stronger-thm}
There exist functions $f:\mathbb N \to \mathbb N$ and $g: \mathbb N \to \mathbb N$ 
such that for all positive integers $k$ and $d$, 
for every graph $G$ and every $A$ subset of the vertices of $G$, 
either $G$ contains $k$ $A$-paths which are pairwise at distance at least $d$, or 
there exists a set $X$ of the vertices of $G$ with $|X|\leq f(k)$ such that 
every $A$-path in $G$
contains a vertex of $B_G(X,g(d))$.
\end{conjecture}

As is typical in coarse graph theory, 
we are especially interested in whether $g$ can be linear in $d$. 
Also, the conjecture may be true with $f(k) = 2k-2$, as in Gallai's theorem.

As a side observation, our proofs of \cref{theorem:main,theorem:main2} are algorithmic. 
For example, \Cref{theorem:main} implies the existence of an algorithm with running time $h(k,d)\cdot n^{\mathcal{O}(1)}$ for some function $h: \mathbb N^2\to \mathbb N$ which, given $k$, $d$, $G$, and $A\subseteq V(G)$ as input, returns either a collection of $k$ 
$A$-paths pairwise at distance at least $d$, or a set $X$ of at most $f(k)$ vertices such that $B_G(X,g(k,d))$ intersects all $A$-paths of $G$. 

\section{Outline of the proof}
One of the recurring proof ideas in Erd\H{o}s--P\'{o}sa theory is the so-called frame technique.
It goes back to 1967, when Simonovits introduced it in an alternative proof of the Erd\H{o}s--P\'{o}sa theorem.
Bruhn, Heinlein, and Joos~\cite{BHJ18} applied the frame technique to get Erd\H{o}s--P\'{o}sa type theorems for even $A$-paths and long $A$-paths. Additionally, they provided a simpler proof of Gallai's theorem, albeit with a weaker bound, 
namely $|X| \leq 4k-2$.
Our proofs of~\Cref{theorem:main,theorem:main2} are likewise based on a coarse adaptation of the frame technique. 
Accordingly, we begin this outline by reviewing the proof of Gallai’s theorem through this approach.

Let $G$ be a graph and $A$ be a subset of the vertices of $G$.
An \defin{$A$-frame} of $G$ is a subcubic forest $H$ contained in $G$ with no isolated vertices, such that $V(H) \cap A$ is exactly the set of leaves of $H$.
Since the null subgraph of $G$ is an $A$-frame, we can pick an inclusion-wise maximal $A$-frame $H$ in $G$.
The main idea is that if $H$
contains a large number of leaves, then it provides a packing of $A$-paths, whereas if $H$ has a bounded number of leaves, then we can find a small set that intersects all $A$-paths.
More formally, we let $c$ denote the number of components of $H$.
First, by a simple induction, see \cref{cor:extracting-paths-from-the-frame}, one can verify that if $H$ contains at least $2k+c-1$ leaves, then it contains $k$ pairwise vertex-disjoint $A$-paths.
Thus, we may assume that $H$ contains at most than $2k+c-2$ leaves.
Then, we define $X$ as all the vertices of degree $1$ or $3$ in $H$.
In a subcubic tree, the number of vertices of degree $3$ equals the number of leaves minus $2$, see~\Cref{lem:bound-branch-vertices}.
In turn, in $H$, the number of vertices of degree $3$ is at most $2k+c-2 - 2c = 2k-c -2$. Therefore, $|X| \leq 4k-4$ and it suffices to argue that every $A$-path in $G$ contains a vertex in $X$.
This follows from the fact that the existence of an $A$-path disjoint from $X$ would allow us to extend the frame and would therefore contradict the maximality of $H$ (see missing details in~\cite[Section~2]{BHJ18}).

In the coarse context, the frame technique works particularly well in the induced setting, i.e. $d=2$.
In this situation, both the proof for cycles~\cite{koreanSODA} and for $A$-paths~\cite{HJ25} are adaptations of the classical frame argument. However, the statements for larger values of $d$ seem more difficult and sometimes false. 
Our adaptation of the frame technique relies on one of the fundamental concepts of coarse graph theory: a fat model of a graph.

Let $G$ and $H$ be graphs. 
A \defin{model} of $H$ in $G$ is a family $(M_x \mid x\in V(H)\cup E(H))$ of connected subgraphs of $G$ such that for all distinct $x,y \in V(H) \cup E(H)$, 
\begin{enumerate}[label=(m\arabic*)]
    \item\label{it:model1} if $x$ and $y$ are incident\footnote{In a graph, a vertex and an edge are \defin{incident} if the vertex is an endpoint of the edge; two edges are \defin{incident} if they share an endpoint; and two vertices are never incident.} in $H$, then
    \begin{itemize}
        \item if $x$ and $y$ are a vertex and an edge, then $V(M_x) \cap V(M_y) \neq \emptyset$;
        \item if $x$ and $y$ are edges sharing an endpoint $z$, then $V(M_x) \cap V(M_y) \subset V(M_z)$;
    \end{itemize}
    \item\label{it:model2} if $x$ and $y$ are not incident in $H$, then $V(M_x) \cap V(M_y) = \emptyset$.
\end{enumerate}
The subgraphs $M_v$ for $v\in V(H)$ are the \defin{branch sets} of the model, and the subgraphs $M_{uv}$ for $uv\in E(H)$ 
are its \defin{branch paths}.
Often, we assume without loss of generality that each branch path $M_{uv}$ is a $V(M_u)$-$V(M_v)$ path in $G$.
Let $\ell$ be a nonnegative integer. 
A model $(M_x \mid x\in V(H)\cup E(H))$ of $H$ in $G$ is \defin{$\ell$-fat} if, for all distinct $x,y\in V(H)\cup E(H)$, 
we have $\dist_G(V(M_x),V(M_y))\geq \ell$, unless $\set{x,y}=\set{v,e}$ where $v\in V(H)$, $e\in E(H)$, and $v$ is incident to $e$ in $H$.

A key idea in our proof is to keep as a frame a fat model of a subcubic forest such that branch sets of vertices of degree at most $2$ contain vertices of $A$. 
As in the classical setting, we fix a maximal frame.
When the frame is small, we aim to extract a hitting set from the frame; therefore, it is essential that each branch set is contained in a ball of small radius.
When the frame is large, we find a packing of $A$-paths that are pairwise far apart in $G$.

The main tool that we develop to extend a frame is the Tripod lemma, which 
we present in a simplified form below. 
See~\Cref{lemma:tripod} for the full statement. 
See also~\Cref{fig:tripod-lemma-illustration}. 
Suppose that we have a connected subgraph $Q$ of $G$ and three vertices $v_1$, $v_2$, $v_3$ in $G$ that~\eqref{lem:tripod-asumpt-dist-xi-xj} are pairwise far apart in $G$, and such that for each $i \in [3]$,~\eqref{lem:tripod-asumpt-dist-xi-Q-large} $v_i$ is not too close and~\eqref{lem:tripod-asumpt-dist-xi-Q-small} not too far from $Q$ in $G$. Then, there are connected subgraphs $Z$, $P_1$, $P_2$, $P_3$ of $G$ such that
\begin{enumerate}
    \item $V_i \in V(P_i)$ and $Z$ intersects $P_i$ for each $i \in [3]$,
    \item $Z$ has bounded radius, 
    \item $V(Z)$ is contained in a bounded radius ball centered on $V(Q)$ in $G$, 
    \item $V(P_i)$ is contained in a bounded radius ball centered on $V(Q) \cup \{v_i\}$ in $G$ for each $i \in [3]$,
    \item $V(P_i)$ and $V(P_j)$ are far apart in $G$ for all distinct $i,j \in [3]$.
\end{enumerate}

\begin{figure}[tp]
  \centering  
  \includegraphics{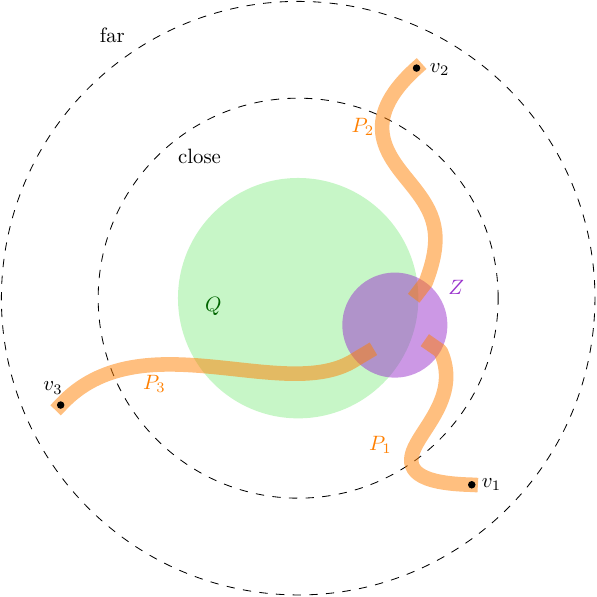}
  \caption{The \nameref{lemma:tripod}.
  The vertices $v_1$, $v_2$, $v_3$ are not too close and not too far from $V(Q)$, and they are far from each other. 
 The subgraphs $Z$, $P_1$, $P_2$, $P_3$ of $G$ are the outcome of the lemma.}  
  \label{fig:tripod-lemma-illustration}
\end{figure}

The \nameref{lemma:tripod} is used in some cases to extend a frame (a fat model of a subcubic forest).
We will assume that there is a path $P$ connecting a vertex in $A$ to a branch path $M_{yz}$ in the frame.
Next, we set as $v_1$ and $v_2$ some specific vertices of $M_{yz}$ and as $v_3$ a vertex of $P$.
Then, $Q$ attaches to $M_{yz}$ in between $v_1$ and $v_2$. Finally, $Z$ will serve as a new branch set of a vertex of degree $3$ in the frame, and $P_1$, $P_2$, $P_3$ will be used to build new branch paths.

The following is an illustrative application of the \nameref{lemma:tripod} that we believe to be of independent interest.
It is a coarse analog of the following basic property of graph minors:
for every subcubic graph $H$, a graph $G$ contains $H$ as a minor if and only if $G$ contains $H$ as a topological minor.

\begin{restatable}{theorem}{topologicalminors}
    \label{lemma:showcase-of-the-tripod-lemma}
    Let $\ell$ be a positive integer, let $G$ be a graph, and let $H$ be a subcubic graph such that $G$ contains a $7\ell$-fat model of $H$. Then $G$ contains a model $\mathcal{N}=(N_x \mid x \in V(H)\cup E(H))$ of $H$ such that 
    \begin{enumerate}
    \item\label{it:thm-top-minor-fat} $\mathcal{N}$ is $\ell$-fat and
    \item\label{it:thm-top-minor-radius} $N_v$ has radius at most $\lfloor1.5\ell\rfloor$ for each $v\in V(H)$.
    \end{enumerate}
\end{restatable} 

The paper is organized as follows. 
\Cref{sec:preliminaries} contains preliminaries.  
In~\Cref{sec:tripod-lemma} we prove the \nameref{lemma:tripod}.
\Cref{sec:augmenting,sec:forests,sec:wrapping} contain the proofs of~\Cref{theorem:main,theorem:main2}.
Namely, in~\Cref{sec:augmenting}, we encapsulate the technical part of the frame-extension step; in~\Cref{sec:forests}, we state and prove some simple observations on subcubic forests; and in~\Cref{sec:wrapping}, we define a frame and wrap up the proofs.
We conclude with the proof of~\Cref{lemma:showcase-of-the-tripod-lemma} in~\Cref{sec:topological-minors}.

\section{Preliminaries}
\label{sec:preliminaries}

By \defin{$\mathbb{N}$}, we denote the set of all positive integers. 
For each $n\in\mathbb{N}$, by \defin{$[n]$},
we denote the set $\set{1,\ldots,n}$.
All graphs in this paper are finite.

Let $G$ be a graph and let $X,Y\subseteq V(G)$. 
An \defin{$X$-$Y$ path} in $G$ is a path from a vertex in $X$ to a vertex in $Y$ with no internal vertices in $X\cup Y$. 
When one of these sets is a single vertex, e.g.\ $X=\set{x}$, we often write an $x$-$Y$ path in $G$ instead of an $\set{x}$-$Y$ path in $G$.
For a path $P$ in $G$ and two vertices $x$ and $y$ of $P$, let \defin{$xPy$} denote the $x$-$y$ subpath of $P$.
For two paths $P$ and $Q$ in $G$ that share exactly one vertex and this vertex is an endpoint of both $P$ and $Q$, we write \defin{$PQ$} to denote the path in $G$ obtained as the concatenation of $P$ and $Q$, i.e.~$P \cup Q$.
For simplicity, we omit repeated elements, e.g.\ when concatenating paths of the form $xPy$ and $yQz$, instead of $xPyyQz$, we write \defin{$xPyQz$}.
Note that we treat edges as two-elements paths.

The \defin{length} of a path $P$, denoted \defin{$\length(P)$}, is the number of edges in $P$.
Let $u$ and $v$ be vertices of $G$. 
The \defin{distance} between $u$ and $v$ in $G$, denoted by \defin{$\dist_G(u,v)$}, is the length of a shortest path between $u$ and $v$ in $G$, or $\infty$ if no such path exists. 

For an integer $r$ and $X \subset V(G)$, we define the \defin{ball of radius $r$ centered on $X$} in $G$ as $\defin{\text{$B_G(X,r)$}} = \{y \in V(G) : \dist_G(X,y) \leq r\}$.
For simplicity, for each $x \in V(G)$, we write \defin{$B_G(x,r)$} meaning $B_G(\{x\},r)$.
Note that we allow radii to be negative, in which case the ball is empty.

The \defin{radius} of a connected graph $G$ is the minimum positive integer $r$ such that there exists $v \in V(G)$ with $V(G) = B_G(v,r)$.
The \defin{degree} of a vertex $v$ in a graph $G$, denoted by \defin{$\deg_G(v)$}, is the number of edges in $G$ incident to $v$. We say that $G$ is \defin{subcubic} if all vertices of $G$ have degree at most $3$.

\section{The tripod lemma}
\label{sec:tripod-lemma}

A key technical ingredient of the frame-extension step in the proof of~\Cref{theorem:main,theorem:main2} is the Tripod Lemma, which we state and prove below.

\begin{lemma}[Tripod Lemma]
\label{lemma:tripod}
Let $G$ be a graph, 
let $v_1$, $v_2$, $v_3$ be vertices of $G$, and let $Q$ be a connected subgraph of $G$. 
Let $\ell$ and $d$ be positive integers such that 
for all distinct $i,j\in[3]$:
\begin{align}
\dist_G(v_i,V(Q)) &\geq \ell,
\label{lem:tripod-asumpt-dist-xi-Q-large}\tag{$\star$}\\
\dist_G(v_i,V(Q))&\leq d,
\label{lem:tripod-asumpt-dist-xi-Q-small}\tag{$\star\star$}\\
\dist_G(v_i,v_j)&\geq 2d.
\label{lem:tripod-asumpt-dist-xi-xj}
\tag{${\star}{\star}{\star}$}
\end{align}
Then, there exist four connected subgraphs $Z$, $P_1$, $P_2$, $P_3$ of $G$ such that
\begin{enumerate}
\item $v_i\in V(P_i)$ and $V(Z)\cap V(P_i)\neq\emptyset$ for each $i\in[3]$, 
\label{lem:tripod-lemma:item:Pi-intersects-Z}
\item $Z\ \textrm{has radius at most $\lfloor1.5\ell\rfloor$}$, 
\label{lem:tripod-lemma:item:Z-has-bounded-radius}
\item $V(Z)\subseteq B_G(V(Q),2\ell-1)$, 
\label{lem:tripod-lemma:item:Z-is-close-to-Q}
\item $V(P_i)  \subseteq B_G(v_i,d-\ell-1) \cup B_G(V(Q),\ell)$ for each $i\in[3]$,
\label{lem:tripod-lemma:item:Pi-does-not-wander}
\item $\dist_G(V(P_i),V(P_j))\geq\ell$ for all distinct $i,j \in [3]$.
\label{lem:tripod-lemma:item:Pi-Pj-are-far-apart}
\end{enumerate}
\end{lemma}

\begin{proof} 
A tuple $(C,\xi,\set{(R_i,w_i,B_i)}_{i\in[3]})$ is a \defin{tripoid} if 
\begin{enumerateAlpha}
\item $C$ is a connected subgraph of $Q$ and $\xi\in[3]$; and 
\label{instance:C-connected}
\end{enumerateAlpha}
for each $i\in[3]$, $R_i$ and $B_i$ are subgraphs of $G$ and $w_i$ is a vertex of $G$ such that
\begin{enumerateAlpha}
\setcounter{enumi}{1}
\item $R_i$ is a $v_i$-$w_i$ path in $G$,
\label{instance:Ri-is-an-xi-yi-path}
\item $\dist_G(V(R_i),V(C))\geq\ell$,
\label{instance:Ri-far-from-C}
\item $\dist_G(w_i,V(C))=\ell$ and
$B_i$ is a shortest $w_i$-$V(C)$ path in $G$,
\label{instance:Bi-is-an-yi-C-path}
\item $V(R_i)\subseteq B_G(v_i,d-\ell-1) \cup B_G(V(Q),\ell)$,
\label{instance:Ri-in-the-ball}
\item $\dist_G(V(R_i),V(R_j))\geq\ell$ for each $j\in[3]\setminus\{i\}$,
\label{instance:distances-R-R}
\item $\dist_G(V(R_i),V(B_j))\geq\ell$ for each $j\in[3]\setminus\set{i,\xi}$.
\label{instance:distances-R-B}
\end{enumerateAlpha}

\begin{figure}[tp]
  \centering  
  \includegraphics[page =1]{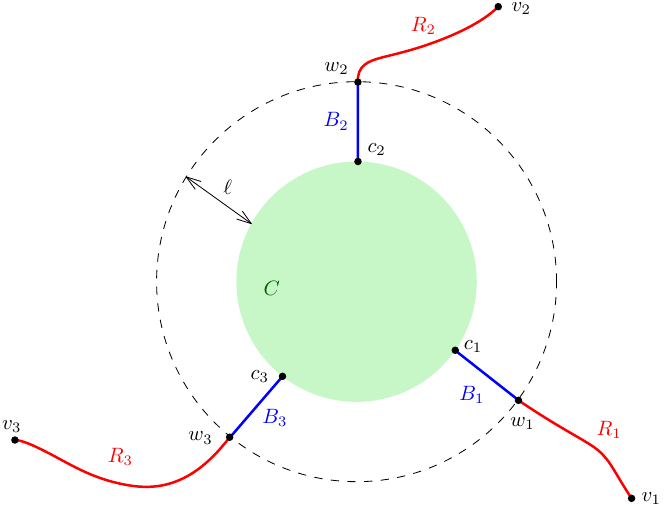}
  \caption{A tripoid $(C,\xi,\set{R_i,w_i,B_i}_{i\in[3]})$.}  
  \label{fig: instance}
\end{figure}

See an illustration in~\Cref{fig: instance}.
The proof strategy is to create a starting tripoid and then keep improving it, with each step of iteration decreasing $|V(C)|$, until we obtain a tripoid from which we can find desired subgraphs $Z$, $P_1$, $P_2$, $P_3$ as in the lemma statement. 

First, we explain how to create an initial tripoid. 
For every $i\in[3]$ we define the following objects.
Let $S_i$ be a shortest path from $v_i$ to $V(Q)$ in $G$.
Let $c_i$ be the endpoint of $S_i$ in $Q$. 
By~\eqref{lem:tripod-asumpt-dist-xi-Q-large}, we have $\dist_G(v_i,c_i)=\dist_G(v_i,V(C))\geq\ell$. 
Let $w_i$ be the vertex of $S_i$ such that $\dist_G(w_i,c_i)=\ell$.
For each $i\in[3]$, let
\[R_i=v_iS_iw_i \ \text{ and } \ B_i = w_iS_ic_i.\]

We claim that $(Q,\xi,\set{(R_i,w_i,B_i)}_{i\in[3]})$ is a tripoid for an arbitrary choice of $\xi\in[3]$.
Since $Q$ is a connected graph, \ref{instance:C-connected} is satisfied. 
Fix $i\in[3]$. 
Item~\ref{instance:Ri-is-an-xi-yi-path} is obviously satisfied by the definition of $R_i$. 
Items~\ref{instance:Ri-far-from-C} and~\ref{instance:Bi-is-an-yi-C-path} hold since $S_i$ is a shortest path in $G$ and by the choice of~$w_i$.
By~\eqref{lem:tripod-asumpt-dist-xi-Q-small}, $S_i$ has length at most $d$.
Since $S_i$ is a $v_i$-$V(Q)$ path in $G$, it follows that $V(R_i)\subseteq V(S_i) \subseteq B_G(v_i,d-\ell-1)\cup B_G(V(Q),\ell)$, so~\ref{instance:Ri-in-the-ball} holds.

Finally, note that for all distinct $i,j\in[3]$, we have 
\begin{align*}
2d &\leq \dist_G(v_i,v_j)&&\textrm{by~\eqref{lem:tripod-asumpt-dist-xi-xj}}\\
&\leq \length(R_i)+\dist_G(V(R_i),V(S_j)) + \length(S_j)\\
&\leq (d-\ell) + \dist_G(V(R_i),V(R_j\cup B_j))+d&&\textrm{by~\eqref{lem:tripod-asumpt-dist-xi-Q-large} and~\eqref{lem:tripod-asumpt-dist-xi-Q-small}}
\end{align*}
which gives
$\dist_G(V(R_i),V(R_j\cup B_j)) \geq \ell$.
This completes the proof of~\ref{instance:distances-R-R} and~\ref{instance:distances-R-B}. 
Thus, $(Q,\xi,\set{(R_i,w_i,B_i)}_{i\in[3]})$ is an instance, as claimed.

In what follows, for a given tripoid $(C,\xi,\set{(R_i,w_i,B_i)}_{i\in[3]})$, for each $i\in [3]$, we will let $c_i$ denote the unique vertex of $B_i$ belonging to $V(C)$.

Now suppose that we 
are given a tripoid $(C,\xi,\set{(R_i,w_i,B_i)}_{i\in[3]})$. 
The plan is to either find subgraphs $Z$, $P_1$, $P_2$, $P_3$ of $G$ 
satisfying the assertion of the lemma or to find another tripoid $(C',\xi',\set{(R_i',w_i',B_i')}_{i\in[3]})$ with $|V(C')|< |V(C)|$. 
This will complete the inductive proof of the lemma.

First, suppose that there is $\alpha\in[3]\setminus\set{\xi}$ such that
$\dist_G(V(R_\alpha),V(B_{\xi}))<\ell$. 
We show that in this case, one can construct subgraphs $Z$, $P_1$, $P_2$, $P_3$ of $G$ satisfying the conclusion of the lemma.
Fix such an index $\alpha$ and $\beta\in[3]$ such that $\set{\alpha,\xi,\beta}=\set{1,2,3}$. Let $S$ be a shortest $V(R_\alpha)$-$V(B_{\xi})$ path in $G$. 
Thus, $\length(S) < \ell$. 
Let $s_\alpha$ and $s_{\xi}$ be the endpoints of $S$ in $V(R_\alpha)$ and $V(B_\xi)$, respectively. 

We define
\begin{align*}
    Z= B_\xi \cup S, \ \ P_\alpha = R_\alpha, \ \ P_\xi = R_\xi, \ \text{ and } \ P_\beta = R_\beta \cup B_\beta \cup C.
\end{align*}
See~\Cref{fig: R-B-close}.
We claim that $Z$, $P_1$, $P_2$, $P_3$ satisfy the assertion of the lemma. 
Each of $B_\xi$, $S$, $R_\alpha$, $R_\xi$, $R_\beta$, $B_\beta$, and $C$ is a connected subgraph of $G$ (by~\ref{instance:C-connected}, \ref{instance:Ri-is-an-xi-yi-path}, and~\ref{instance:Bi-is-an-yi-C-path}).
Since $s_\xi$ is a common vertex of $B_\xi$ and $S$, $Z$ is connected.
Clearly, $P_\alpha$ and $P_\xi$ are connected.
Since $w_\beta$ is a common vertex of $R_\beta$ and $B_\beta$, and $c_\beta$ is a common vertex of $B_\beta$ and $C$, $P_\beta$ is connected.

By~\ref{instance:Ri-is-an-xi-yi-path}, $v_i \in V(R_i)\subseteq V(P_i)$ for each $i \in [3]$.
Moreover, 
\begin{align*}
    s_\alpha \in V(R_\alpha) \cap V(S) &\subseteq V(P_\alpha) \cap V(Z),\\
    w_\xi \in V(R_\xi) \cap V(B_\xi) &\subset V(P_\xi) \cap V(Z),\\
    c_\xi \in V(C) \cap V(B_\xi) &\subset V(P_\beta) \cap V(Z).
\end{align*}
Thus,~\ref{lem:tripod-lemma:item:Pi-intersects-Z} holds.

Recall that $Z$ is the union of two paths that share a vertex: $B_\xi$ of length $\ell$ by~\ref{instance:Bi-is-an-yi-C-path} and $S$ of length less than $\ell$ by the case assumption. 
This implies that $Z$ has radius at most $\ell$.
Thus,~\ref{lem:tripod-lemma:item:Z-has-bounded-radius} holds. 

\begin{figure}[tp]
  \centering  
  \includegraphics[page =1]{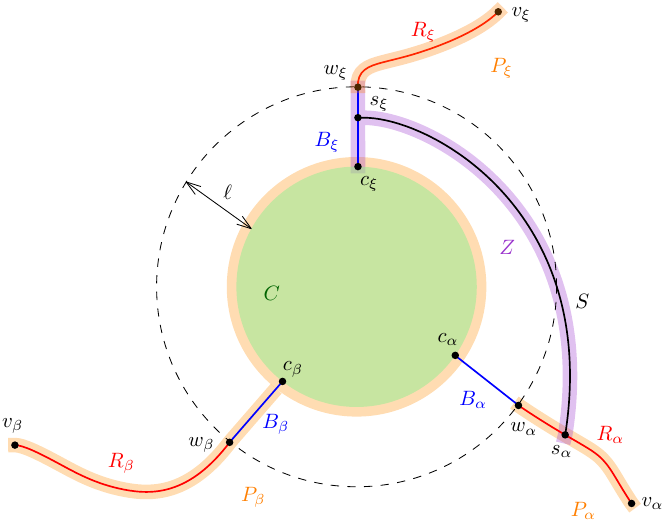}
  \caption{Construction of $Z$, $P_1$, $P_2$, $P_3$ when $\dist_G(V(R_\alpha), V(B_\xi)) < \ell$.} 
  \label{fig: R-B-close}
\end{figure}

We have $V(B_\xi) \subseteq B_G(c_\xi,\ell)$ (by~\ref{instance:Bi-is-an-yi-C-path}) 
and since $s_\xi$ is a vertex of $B_\xi$, we have $V(S) \subset B_G(c_\xi,2\ell-1)$.
Altogether, $V(Z) \subset B_G(c_\xi,\max\{\ell,2\ell-1\}) = B_G(c_\xi,2\ell-1)$.
Since $c_\xi \in V(C)$ and $C$ is a subgraph of $Q$ (by~\ref{instance:C-connected}), $V(Z) \subset B_G(V(Q),2\ell-1)$. 
Thus~\ref{lem:tripod-lemma:item:Z-is-close-to-Q} holds.

Observe that 
\begin{align*}
V(P_\alpha) &= V(R_\alpha) \subseteq B_G(v_\alpha,d-\ell-1) \cup B_G(V(Q),\ell)&&\textrm{by~\ref{instance:Ri-in-the-ball};}\\
V(P_\xi) &= V(R_\xi) \subseteq B_G(v_\xi,d-\ell-1) \cup B_G(V(Q),\ell)&&\textrm{by~\ref{instance:Ri-in-the-ball};}\\
V(P_\beta)&= V(R_\beta) \cup V(B_\beta) \cup V(C)\\
&\subseteq B_G(v_\beta,d-\ell-1) \cup B_G(V(Q),\ell),&&\textrm{by~\ref{instance:Ri-in-the-ball},~\ref{instance:Bi-is-an-yi-C-path},~and~\ref{instance:C-connected};}
\end{align*}
so~\ref{lem:tripod-lemma:item:Pi-does-not-wander} holds.

Finally, note that 
\begin{align*}
\dist_G(V(P_\alpha),V(P_\xi\cup P_\beta)) &= \dist_G(V(R_\alpha),V(R_\xi\cup R_\beta\cup B_\beta\cup C) ) \geq \ell&&\textrm{by~\ref{instance:distances-R-R},~\ref{instance:distances-R-B}, and~\ref{instance:Ri-far-from-C};}\\
\dist_G(V(P_\xi),V(P_\beta)) &= \dist_G(V(R_\xi),V(R_\beta\cup B_\beta\cup C) \geq \ell&&\textrm{by~\ref{instance:distances-R-R},~\ref{instance:distances-R-B}, and~\ref{instance:Ri-far-from-C}.}
\end{align*}
Thus~\ref{lem:tripod-lemma:item:Pi-Pj-are-far-apart} holds. 
This completes the proof that $Z$, $P_1$, $P_2$, $P_3$ satisfy the assertion of the lemma.

From now on, we assume that the given tripoid $(C,\xi,\set{(R_i,w_i,B_i)}_{i\in[3]})$ satisfies
\begin{enumerateAlpha'}
\setcounter{enumi}{6}
\item $\dist_G(V(R_i),V(B_j))\geq\ell$ for all distinct $i,j\in[3]$.
\label{instance':distances-R-B}
\end{enumerateAlpha'}
In other words the value of $\xi$ is now irrelevant. 

Next, suppose that there are distinct $\alpha,\beta\in[3]$ such that
$\dist_G(V(B_\alpha),V(B_\beta))< \ell$. Again, we show that in this case, one can directly construct $Z$ and $P_1$, $P_2$, $P_3$ satisfying the conclusion of the lemma.
Let $\alpha$ and $\beta$ be such indices and fix $\gamma\in[3]$ such that $\set{\alpha,\beta,\gamma}=\set{1,2,3}$. 
Let $S$ be a shortest $V(B_\alpha)$-$V(B_\beta)$ path in $G$. 
Thus, $\length(S)< \ell$. 

\begin{figure}[tp]
  \centering  
  \includegraphics[page =1]{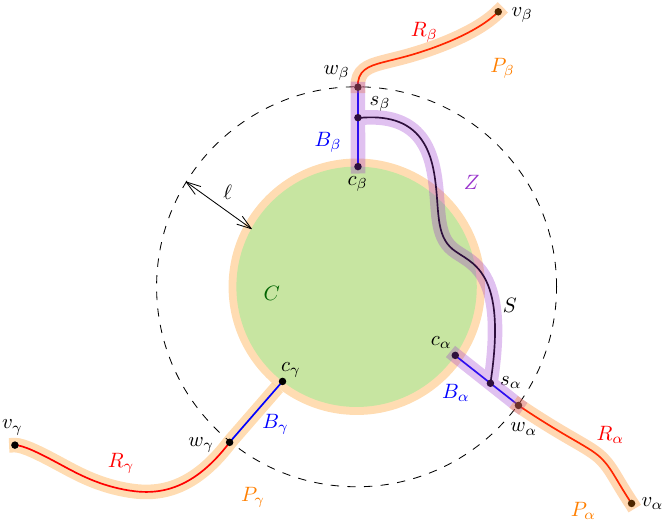}
  \caption{Construction of $Z$, $P_1$, $P_2$, $P_3$ when $\dist_G(V(B_\alpha), V(B_\beta)) < \ell$.} 
  \label{fig: B-B-close}
\end{figure}

We define
\begin{align*}
    Z= B_\alpha \cup B_\beta \cup S, \ \ P_\alpha = R_\alpha, \ \ P_\beta = R_\beta, \ \text{ and } \ P_\gamma = R_\gamma \cup B_\gamma \cup C.
\end{align*}
See~\Cref{fig: B-B-close}.
We claim that $Z$, $P_1$, $P_2$, $P_3$ satisfy the assertion of the lemma. 
Each of $B_\alpha$, $B_\beta$, $S$, $R_\alpha$, $R_\beta$, $R_\gamma$, $B_\gamma$, and $C$ is a connected subgraph of $G$ (by~\ref{instance:C-connected}, \ref{instance:Ri-is-an-xi-yi-path}, and~\ref{instance:Bi-is-an-yi-C-path}).
Since $s_\alpha$ is a common vertex of $B_\alpha$ and $S$, and $s_\beta$ is a common vertex of $B_\beta$ and $S$, $Z$ is connected.
Clearly, $P_\alpha$ and $P_\beta$ are connected.
Since $w_\gamma$ is a common vertex of $R_\gamma$ and $B_\gamma$, and $c_\gamma$ is a common vertex of $B_\gamma$ and $C$, $P_\gamma$ is connected.

By~\ref{instance:Ri-is-an-xi-yi-path}, $v_i \in V(R_i)\subseteq V(P_i)$ for each $i \in [3]$.
Moreover, 
\begin{align*}
    w_\alpha \in V(R_\alpha) \cap V(S) &\subseteq V(P_\alpha) \cap V(Z),\\
    w_\beta \in V(R_\beta) \cap V(B_\xi) &\subset V(P_\beta) \cap V(Z),\\
    c_\alpha \in V(C) \cap V(B_\alpha) &\subset V(P_\gamma) \cap V(Z).
\end{align*}
Thus,~\ref{lem:tripod-lemma:item:Pi-intersects-Z} holds.

Recall that $Z$ is the union of three paths: $B_\alpha$, $B_\beta$ and $S$. 
Since both $B_\alpha$ and $B_\beta$ have length $\ell$, and the length of $S$ is less than $\ell$, 
we obtain that $Z$ has radius at most $\lceil\frac{\ell-1}{2}\rceil+\ell = \lfloor \frac{\ell}{2} \rfloor + \ell = \lfloor 1.5 \ell \rfloor$.
Thus,~\ref{lem:tripod-lemma:item:Z-has-bounded-radius} holds.

We have $V(B_\alpha) \subseteq B_G(c_\alpha,\ell)$ and $V(B_\beta) \subset B_G(c_\beta,\ell)$ (by~\ref{instance:Bi-is-an-yi-C-path}).
It follows that $V(S) \subset B_G(c_\alpha,2\ell-1)$.
Since $c_\alpha,c_\beta \in V(C)$ and $C$ is a subgraph of $Q$ (by~\ref{instance:C-connected}), $V(Z) \subset B_G(V(Q),\max\set{\ell,2\ell-1}) = B_G(V(Q),2\ell-1)$. Thus, \ref{lem:tripod-lemma:item:Z-is-close-to-Q} holds.

Note that 
\begin{align*}
V(P_\alpha) &= V(R_\alpha) \subseteq B_G(v_\alpha,d-\ell-1) \cup B_G(V(Q),\ell)&&\textrm{by~\ref{instance:Ri-in-the-ball};}\\
V(P_\beta) &= V(R_\beta) \subseteq B_G(v_\beta,d-\ell-1) \cup B_G(V(Q),\ell)&&\textrm{by~\ref{instance:Ri-in-the-ball};}\\
V(P_\gamma)&= V(R_\gamma) \cup V(B_\gamma) \cup V(C)\\
&\subseteq B_G(v_\gamma,d-\ell-1) \cup B_G(V(Q),\ell),&&\textrm{by~\ref{instance:Ri-in-the-ball},~\ref{instance:Bi-is-an-yi-C-path}, and~\ref{instance:C-connected}.}
\end{align*}
Thus,~\ref{lem:tripod-lemma:item:Pi-does-not-wander} holds.

Finally, note that 
\begin{align*}
\dist_G(V(P_\alpha),V(P_\beta\cup P_\gamma)) &= \dist_G(V(R_\alpha),V(R_\beta\cup R_\gamma\cup B_\gamma\cup C) ) \geq \ell&&\textrm{by~\ref{instance:distances-R-R},~\ref{instance':distances-R-B}, and~\ref{instance:Ri-far-from-C};}\\
\dist_G(V(P_\beta),V(P_\gamma)) &= \dist_G(V(R_\beta),V(R_\gamma\cup B_\gamma\cup C) \geq \ell&&\textrm{by~\ref{instance:distances-R-R},~\ref{instance':distances-R-B}, and~\ref{instance:Ri-far-from-C}.}
\end{align*}
Thus,~\ref{lem:tripod-lemma:item:Pi-Pj-are-far-apart} holds. 
This completes the proof that $Z$, $P_1$, $P_2$, $P_3$ are the desired outcome of the lemma.

Therefore, we may assume from now on that our tripoid 
$\calT=(C,\xi,\set{(R_i,w_i,B_i)}_{i\in[3]})$ satisfies
\begin{enumerateAlpha}
\setcounter{enumi}{7}
\item $\dist_G(V(B_i),V(B_j)) \geq \ell$ for all distinct $i,j \in [3]$. 
\label{instance:distances-B-B}
\end{enumerateAlpha}
Recall that $c_1$, $c_2$, and $c_3$ are vertices of $C$. 
Note that by~\ref{instance:distances-B-B} these are three distinct vertices. 
We claim that there is $\alpha \in [3]$ such that if $\{\alpha,\beta,\gamma\} = \{1,2,3\}$, then $c_\beta$ and $c_\gamma$ lie in the same component of $C - c_\alpha$.
If $c_2$ and $c_3$ lie in the same component of $C - c_1$, then we set $\alpha = 1$.
Thus, suppose that $c_2$ and $c_3$ lie in different components of $C - c_1$.
Since $C$ is connected, we can fix a $c_2$-$c_3$ path $P$ in $C$.
It follows that $c_1$ is an internal vertex of $P$.
Note that $P - c_2$ is a path in $C$.
Moreover, it contains both $c_1$ and $c_3$.
It follows that $c_1$ and $c_3$ are in the same component of $C - c_2$ and we can set $\alpha = 2$.

Therefore, we fix $\alpha,\beta,\gamma\in[3]$ such that 
$\set{\alpha,\beta,\gamma}=\{1,2,3\}$, and
$c_\beta$ and $c_\gamma$ lie in the same component of $C-c_\alpha$. 
Let $D$ denote the component of $C-c_\alpha$ containing $c_\beta$ and $c_\gamma$. 

Note that by~\ref{instance:Ri-far-from-C}, we have
\[
\dist_G(w_\alpha,V(D))\geq \dist_G(w_\alpha,V(C)) = \ell.
\]

Consider now the specific case where $\dist_G(w_\alpha,V(D))=\ell$ (see \cref{fig: tripod-ind} left). 
Let $S$ be a shortest $w_\alpha$-$V(D)$ path in $G$. 
Let 
\begin{align*}
&R_\alpha'=R_\alpha,\ w_\alpha'=w_\alpha,\ B_\alpha'=S, &&\textrm{and}\\
&(R_i',w_i',B_i')=(R_i,w_i,B_i)&&\textrm{for each $i\in\set{\beta,\gamma}$.}
\end{align*}
We claim that $(D, \alpha,\set{(R_i',w_i',B_i')}_{i\in[3]})$ is a tripoid. Since $|V(D)| < |V(C)|$ (as $V(D)\subseteq V(C)$ and $c_\alpha\not\in V(D)$), 
it will conclude the proof in this case. 

\begin{figure}[tp]
  \centering  
  \includegraphics{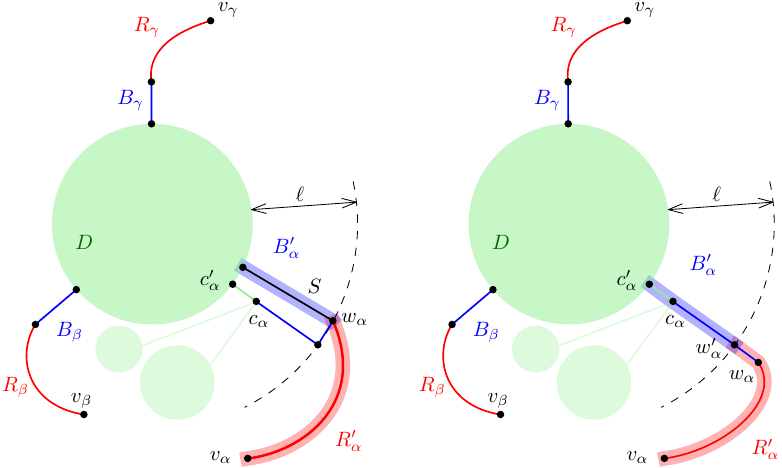}
  \caption{Construction of $(D, \alpha,\set{(R_i',w_i',B_i')}_{i\in[3]})$ within the proof of the \nameref{lemma:tripod}. 
  On the left, we depict the case where $\dist_G(w_\alpha,V(D)) = \ell$ and on the right, we depict the case where $\dist_G(w_\alpha,V(D)) > \ell$.
  We highlight in blue the new path $B_\alpha'$ and in red the new path $R_\alpha'$.
  In the depicted cases, $c_\alpha$ is a cut-vertex of $C$.
  The light green bubbles are components of $C - c_\alpha$, one of them is $D$.
  } 
  \label{fig: tripod-ind}
\end{figure}

Since $D \subseteq C \subseteq Q$ (by~\ref{instance:C-connected} for $\calT$) and $D$ is connected, \ref{instance:C-connected} holds. 
Since $R_i'=R_i$ and $w_i'=w_i$ for all $i\in[3]$, 
\ref{instance:Ri-is-an-xi-yi-path},~\ref{instance:Ri-far-from-C},~\ref{instance:Ri-in-the-ball}, and~\ref{instance:distances-R-R} still hold.
Since $c_\beta,c_\gamma \in V(D)$, since $D$ is a subgraph of $C$, and 
since we have assumed in this case
$\dist_G(w_\alpha,V(D))=\ell$, we have \[
\dist_G(w_i',V(D)) = \ell,
\]
for each $i\in[3]$. 
Recall that by~\ref{instance:Bi-is-an-yi-C-path} for $\calT$, 
$B_i'=B_i$ is a shortest $w_i$-$V(D)$ path in $G$ for each $i\in\set{\beta,\gamma}$. 
Moreover, $S=B_\alpha'$ is defined to be a shortest $w_\alpha$-$V(D)$ path in $G$.
All this implies that~\ref{instance:Bi-is-an-yi-C-path} holds.

Finally, \ref{instance:distances-R-B} holds as we assumed that $\calT$ satisfies a stronger version of it, namely~\ref{instance':distances-R-B}, 
and the only new path
$B_\alpha'$ is exempted.
This completes the proof that $(D,\alpha,\set{(R_i',w_i',B_i')_{i\in[3]}})$ is a tripoid.

Now consider the case that $\dist_G(w_\alpha,V(D))>\ell$ (see \cref{fig: tripod-ind} right). 
Let 
\begin{align*}
w_\alpha'&\qquad 
\textrm{be the neighbor of $w_\alpha$ in $B_\alpha$ and let}\\
c_\alpha'&\qquad 
\textrm{be a neighbor of $c_\alpha$ in $V(D)$.}
\end{align*}
Note that $c_\alpha'$ exists as $C$ is connected and $D$ is a non-null component of $C-c_\alpha$.
Let
\begin{align*}
&R_\alpha'=v_\alpha R_\alpha w_\alpha w_\alpha',\  B_\alpha'=w_\alpha'B_\alpha c_\alpha c_\alpha', &&\textrm{and}\\
&(R_i',w_i',B_i')=(R_i,w_i,B_i)&&\textrm{for each $i\in\set{\beta,\gamma}$.}
\end{align*}
We claim that $(D,\alpha,\set{(R_i',w_i',B_i')}_{i\in[3]})$ is an tripoid. Since $|V(D)| < |V(C)|$ (as $c_\alpha\not\in V(D)$), 
it will conclude the proof in this case. 

Since $D \subseteq C \subseteq Q$ (by~\ref{instance:C-connected} for $\calT$) and $D$ is connected, \ref{instance:C-connected} holds. 
By~\ref{instance:Ri-far-from-C} and~\ref{instance:Bi-is-an-yi-C-path} for $\calT$, we have $\dist_G(V(R_\alpha),V(C)) \geq \ell$ and $\dist_G(w_\alpha',V(C)) = \ell-1$, hence $R_\alpha'$ is a $v_\alpha$-$w_\alpha'$ path in $G$.
Since $R_i'=R_i$ are still $v_i$-$w_i'$ paths in $G$, for each $i\in\set{\beta,\gamma}$, so \ref{instance:Ri-is-an-xi-yi-path} holds.
Since $\calT$ satisfies~\ref{instance:Ri-far-from-C}, and $D\subseteq C$, and $\bigcup_{i\in[3]} V(R_i') = \bigcup_{i\in[3]} V(R_i) \cup \set{w_\alpha'}$, 
the only thing we need to verify to establish~\ref{instance:Ri-far-from-C} is 
$\dist_G(w_\alpha',V(D))\geq \ell$. 
From the case assumption, we obtain
\[
\dist_G(w_\alpha',V(D)) \geq \dist_G(w_\alpha,V(D))-1 \geq \ell+1-1=\ell.
\]
Thus~\ref{instance:Ri-far-from-C} holds. 
Note that by the previous display and the fact that $B_\alpha'$ is a $w_\alpha'$-$V(D)$ path of length $\ell$, we obtain that 
$\dist_G(w_\alpha',V(D))=\ell$. 
In particular, $B_\alpha'$ is a shortest $w_\alpha'$-$V(D)$ path in $G$.
This together with the fact that $c_\beta,c_\gamma\in V(D) \subset V(C)$ implies that~\ref{instance:Bi-is-an-yi-C-path} holds.
Again since $\bigcup_{i\in[3]} V(R_i') =  \bigcup_{i\in[3]} V(R_i') \cup \set{w_\alpha'}$, to verify~\ref{instance:Ri-in-the-ball}, we only need to note that $w_\alpha'\in V(B_\alpha) \subseteq B_G(V(C),\ell) \subseteq B_G(V(Q),\ell)$ (by~\ref{instance:Bi-is-an-yi-C-path} and~\ref{instance:C-connected} for $\calT$). 
Similarly, \ref{instance:distances-R-R} follows since by~\ref{instance':distances-R-B} for $\calT$, we have
\[\dist_G(w_\alpha',V(R_\beta')\cup V(R_\gamma'))\geq\dist_G(V(B_{\alpha}),V(R_\beta')\cup V(R_\gamma'))\geq\ell.
\]
Finally, by~\ref{instance:distances-B-B} for $\calT$,
\[
\dist_G(w_\alpha',V(B_\beta')\cup V(B_\gamma'))\geq \dist_G(V(B_\alpha),V(B_\beta) \cup V(B_\gamma))\geq\ell.
\]
Therefore,~\ref{instance:distances-R-B} follows as $\alpha$ plays the role of the special index $\xi$.
This concludes the proof that $(D,\alpha,\set{(R_i',w_i',B_i')_{i\in[3]}})$ is a tripoid, and thus also finishes the proof of the lemma.
\end{proof}

\section{Augmenting the model}
\label{sec:augmenting}
The main result of this section is \cref{lem:augmenting-the-model}.
It is the main part of the frame-extension step in the final proofs of~\Cref{theorem:main,theorem:main2}.
\cref{lem:augmenting-the-model} roughly states that for each positive integer $\ell$, there exists a larger integer $\ell'$ such that given an $\ell'$-fat model $\mathcal M$ of a subcubic graph $H$ in a graph $G$, the following holds. 
Presume there is a path $P$ that is far enough from all the branch sets of $\calM$ and all the branch paths of $\calM$ except one $M_{yz}$ to which $P$ is close. 
Then one can construct an $\ell$-fat model $\mathcal N$ of one of the two graphs $H'$ and $H''$ depicted in \cref{fig: ModelH}
The proof of \cref{lem:augmenting-the-model} fundamentally relies on the \nameref{lemma:tripod}.

\begin{figure}[tp]
  \centering  
  \includegraphics{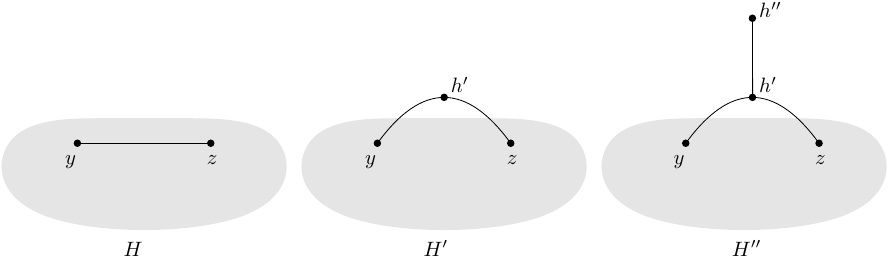}
  \caption{Given a graph $H$ and an edge $yz$ in $H$, 
  the graph $H'$ is obtained from $H$ by subdividing $yz$ once. 
  If $h'$ is the new vertex, then the graph $H''$ is obtained from $H'$ by attaching to $h'$ a new vertex $h''$.
  These extensions of $H$ are considered in~\cref{lem:augmenting-the-model}.}
  \label{fig: ModelH}
\end{figure}

Let $G$ and $H$ be graphs and let $\calM=(M_x \mid x\in V(H)\cup E(H))$ be a model of $H$ in $G$. 
Let $yz$ be an edge of $H$.
Even if $\calM$ is $\ell$-fat for some large integer $\ell$, we do not have control over how the path $M_{yz}$ behaves with respect to $M_y$.
Namely, $M_{yz}$ may approach $M_y$ arbitrarily many times before ultimately leaving for the other endpoint in $M_z$.
Such behavior is problematic, therefore to deal with it, we introduce the notion of $\ell$-clean models, and we prove that by losing at most a small factor of fatness, we can assume that our model is clean, see~\Cref{lemma:fat-to-clean}. 

Let $\calM=(M_x \mid x\in V(H)\cup E(H))$ be a model of a graph $H$ in a graph $G$ and $\ell$ be a nonnegative integer. 
We say that $\calM$ is \defin{simple} if for every $uv \in E(H)$, the branch path $M_{uv}$ is a $V(M_u)$-$V(M_v)$ path in $G$.
We say that $\calM$ is \defin{$\ell$-clean} 
if it is simple and for all $v\in V(H)$ and $e\in E(H)$ such that $v$ is incident to $e$ in $H$, for each $i \in\set{0,\ldots,\ell}$, there is exactly one vertex of $M_e$ at distance $i$ to $M_v$ in $G$.
See~\cref{fig: clean}.

\begin{figure}[tp]
  \centering  
  \includegraphics{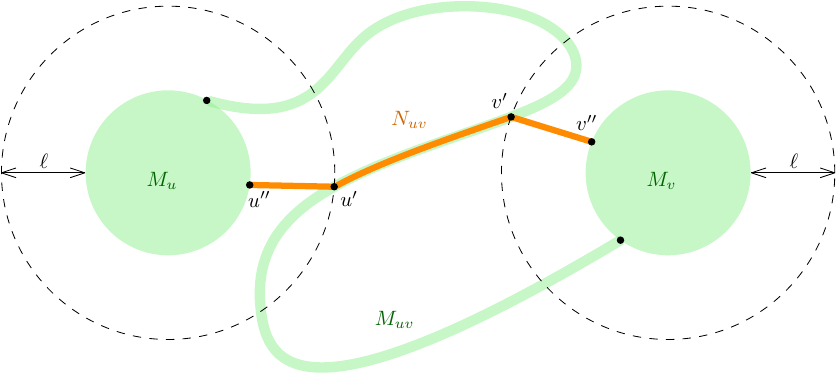}
  \caption{Let $H$ be a two-vertex path $uv$. The model $\{M_u,M_v,M_{uv}\}$ depicted in green is not $\ell$-clean.
  On the other hand, replacing $M_{uv}$ by $N_{uv}$, depicted in orange, gives an $\ell$-clean model. Notation is consistent with the proof of~\cref{lemma:fat-to-clean}, which shows how to transform a fat model into a fat and clean model.}
  \label{fig: clean}
\end{figure}

\begin{lemma}
\label{lemma:fat-to-clean}
Let $q$ and $\ell$ be integers with $q\geq \ell \geq 1$.
Let $G$ and $H$ be graphs, and let $\calM=(M_x \mid x\in V(H)\cup E(H))$ be a $(q+2\ell)$-fat model of $H$ in $G$. 
Then there exists a model $\calN=(N_x \mid x\in V(H)\cup E(H))$ of $H$ in $G$ such that 
\begin{enumerate}
\item $N_v = M_v$ for each $v\in V(H)$,
\label{lemma:fat-to-clean:item:branchsets}
\label{lemma:fat-to-clean:item:paths}
\item $\calN$ is $q$-fat, and 
\label{lemma:fat-to-clean:item:fat}
\item $\calN$ is $\ell$-clean.
\label{lemma:fat-to-clean:item:clean}
\end{enumerate}
\end{lemma}

\begin{proof} 
For each $v\in V(H)$, let $N_v=M_v$ and 
for each $uv\in E(H)$, we define 
$N_{uv}$ as follows.  
Note that $M_{uv}$ is a connected graph containing a vertex of $B_G(V(M_u),\ell)$ and a vertex of $B_G(V(M_v),\ell)$, and as $\calM$ is $(q+2\ell)$-fat, these two balls are disjoint.
In particular, there exists a $B_G(V(M_u),\ell)$-$B_G(V(M_v),\ell)$ path $W_{uv}$ in $M_{uv}$.
Let $u'$ denote the endpoint of $W_{uv}$ in $B_G(V(M_u),\ell)$, 
let $W_u$ be a shortest $V(M_u)$-$u'$ path in $G$, and let $u''$ be the endpoint of $W_u$ in $V(M_u)$.
Similarly, let $v'$ denote the endpoint of $W_{uv}$ in $B_G(V(M_v),\ell)$, 
let $W_v$ be a shortest $V(M_v)$-$v'$ path in $G$, and let $v''$ be the endpoint of $W_v$ in $V(M_v)$.
We define
\[
N_{uv} = u'' W_u u' W_{uv} v' W_vv''.
\]
Note that $N_{uv}$ is a $V(N_u)$-$V(N_v)$ path in $G$.
This completes the construction of $\calN$.

We claim that $\calN$ is a model of $H$ in $G$ satisfying the assertion of the lemma. 
First note that for all $v\in V(H)$ and $e\in E(H)$ such that $v$ is incident to $e$, we have $V(N_v) \cap V(N_e)\neq\emptyset$. 
For each $uv \in E(H)$, $W_u$ and $W_v$ both have length $\ell$ and $W_{uv}$ is a subgraph of $M_{uv}$, hence 
\begin{equation}
V(N_{uv}) \subseteq B_G(V(M_{uv}),\ell).
\label{eq:fat-to-clean-edges}
\end{equation}
For all $x,y\in V(H)\cup E(H)$ such that $\{x,y\} \neq \{v,e\}$ where $v \in V(H)$, $e \in E(H)$, and $v$ is incident to $e$ in $H$, we have 
\begin{align*}
\dist_G(V(N_x),V(N_y)) 
&\geq \dist_G(B_G(V(M_x),\ell),B_G(V(M_y),\ell)) && \text{by \ref{lemma:fat-to-clean:item:branchsets} and~\eqref{eq:fat-to-clean-edges}}\\
&\geq \dist_G(V(M_x),V(M_y)) - 2\ell\\
&\geq q+2\ell-2\ell = q && \text{as $\calM$ is $(q+2\ell)$-fat}.
\end{align*}
Thus, $\calN$ is a $q$-fat model of $H$ in $G$ and~\ref{lemma:fat-to-clean:item:fat} follows.

Item~\ref{lemma:fat-to-clean:item:branchsets} follows by construction. 
Also by construction, $\calN$ is simple.
For each $uv \in E(H)$, since $W_{uv}$ is internally disjoint from $B_G(V(M_u),\ell)\cup B_G(V(M_v),\ell)$, we have $V(N_{uv}) \cap B_G(V(M_u),\ell)= V(W_u)$. 
Since $W_u$ is a shortest path from $V(M_u)$ to a vertex at distance $\ell$ from $V(M_u)$, 
we conclude that $\mathcal{N}$ is $\ell$-clean.
This completes the proof of~\ref{lemma:fat-to-clean:item:clean}.
\end{proof}

We now have everything in hand to state and prove the main result of the section.

\begin{lemma}\label{lem:augmenting-the-model}
Let $\ell$ be a positive integer. 
Let $G$ be a graph, let $H$ be a subcubic graph, and let $\calM = (M_x \mid x\in V(H)\cup E(H))$ be an $8\ell$-fat and $4\ell$-clean model of $H$ in $G$.
Let $a \in V(G)$ and $yz \in E(H)$, and let $P$ be an $a$-$B_G(V(M_{yz}),4\ell)$ path in $G$ with
\begin{equation}\label{eq:P-far-from-the-model-vertices}
     \dist_G\left(V(P), \bigcup_{x\in V(H)} V(M_x)\right) \geq 8\ell \tag{$\dagger$}
\end{equation}
and
\begin{equation}\label{eq:P-far-from-the-model-edges}
     \dist_G\left(V(P), \bigcup_{x\in E(H) \setminus\{yz\}} V(M_x)\right) \geq 4\ell. \tag{$\dagger\dagger$}
\end{equation}
Let $H'$ be the graph obtained from $H$ by subdividing the edge $yz$ once and let $h'$ denote the new vertex, and let $H''$ be the graph obtained from $H'$ by attaching a new vertex $h''$ adjacent only to $h'$ in $H''$.
Then either 
there is an $\ell$-fat model $(N_x \mid x\in V(H')\cup E(H'))$ of $H'$ in $G$ such that 
\begin{enumerateAlpha'}
\item $N_x=M_x$ for each $x \in V(H)\cup E(H)\setminus\set{yz}$ and  \label{item:modelH'-other-branchsets}
\item $a\in V(N_{h'})$ and $N_{h'}$ has radius at most $4\ell$,  \label{item:modelH'-h'}
\end{enumerateAlpha'}
or there is an $\ell$-fat model $(N_x \mid x\in V(H'')\cup E(H''))$ of $H''$ in $G$ such that 
\begin{enumerateAlpha''}
\item $N_x=M_x$ for each $x \in V(H)\cup E(H)\setminus\set{yz}$,  \label{item:modelH''-other-branchsets}
\item $N_{h'}$ has radius at most $4\ell$, and \label{item:modelH''-h'}
\item $V(N_{h''}) = \{a\}$.  \label{item:modelH''-h''}
\end{enumerateAlpha''}
\end{lemma}
\begin{proof}
    Let $w$ be the endpoint of $P$ in $B_G(V(M_{yz}), 4\ell)$. 
    Note that as internal vertices of $P$ are not in $B_G(V(M_{yz}), 4\ell)$, either $w\neq a$ and $\dist_G(V(P),V(M_{yz}))=4\ell$, or $w=a$, $V(P) = \{a\}$ and $a \in B_G(V(M_{yz}), 4\ell)$.
    By~\eqref{eq:P-far-from-the-model-vertices} we have,
    \begin{equation}
    \label{eq:dist-w-My-Mz}
    \dist_G(w,V(M_y\cup M_z))\geq 8\ell.
    \end{equation}
    For each $x \in \{y,z\}$, 
    we define the following.
    Let $v_x$ be the endpoint of $M_{yz}$ in $M_x$.
    Let $q_x$ be the first vertex of $M_{yz}$ starting from $v_x$ such that $\dist_G(w,q_x) = 4\ell$. 
    Note that $q_x$ is well-defined by~\eqref{eq:dist-w-My-Mz}. 
    Let $Q_x$ be the $v_x$-$q_x$ subpath of $M_{yz}$ and let $W_x$ be a $w$-$q_x$ path of length $4\ell$ in $G$. Refer to \cref{fig: augmenting_config}. 

    \begin{figure}[tp]
    \centering  
    \includegraphics{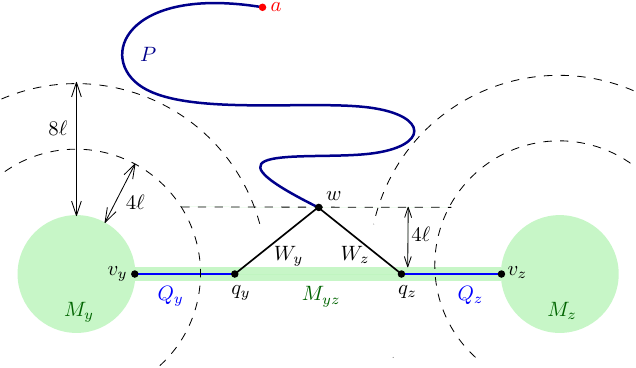}
    \caption{A setup in the proof of \cref{lem:augmenting-the-model}.
    Note that it is possible that $q_y = q_z$, and that in general, $q_y$ and $q_z$ do not necessarily belong to the balls of radius $8\ell$ around $M_y$ and $M_z$.
    }  
    \label{fig: augmenting_config}
    \end{figure}
    
    Note that by~\eqref{eq:dist-w-My-Mz}, for each $x \in \{y,z\}$,
    \begin{align*}
        \dist_G(q_x, V(M_y \cup M_z)) 
        &\geq \dist_G(w,V(M_y\cup M_z)) - \dist_G(w,q_x)
        \geq 8\ell - 4\ell = 4\ell.
    \end{align*}
    For future reference, we write
    \begin{equation}
        \dist_G(\{q_y,q_z\}, V(M_y \cup M_z)) \geq  4\ell.
    \label{eq:q-far-from-M}
    \end{equation}
    We have $\dist_G(q_y,V(M_z)) \geq 4\ell$ (by~\eqref{eq:q-far-from-M}) and $Q_y = q_yM_{yz}v_y$, hence since $\calM$ is $4\ell$-clean, $B_G(V(M_z), 4\ell) \cap V(Q_y) \subset \{q_y\}$.
    This and a symmetric argument gives,
    \begin{equation}
    \label{eq:Q-far-from-M}
        \dist_G(V(Q_y), V(M_z)) \geq 4\ell \ \ \text{ and } \ \ \dist_G(V(Q_z), V(M_y)) \geq 4\ell.
    \end{equation}
    Additionally, by the definitions of $Q_y$ and $Q_z$,
    \begin{equation}\label{eq:aPw-far-Myz}
        \dist_G(V(P), V(Q_y \cup Q_z)) \geq 4\ell.
    \end{equation}
    
    We consider three cases depending on the distances in $G$ between $Q_y$ and $Q_z$, and between $a$ and $w$.

    \textcolor{red}{Case 1.} $\dist_G(V(Q_y),V(Q_z)) \geq \ell$ and $\dist_G(a,w) \leq 2\ell-1$.

    Let $P'$ be a shortest $a$-$w$ path in $G$. 
    For every $x\in V(H')\cup E(H')$ let
    \[
    N_x = \begin{cases}
    M_x&\textrm{if $x \in V(H)\cup E(H)\setminus\set{yz}$,}\\
    W_y\cup W_z \cup P'&\textrm{if $x=h'$,}\\
    Q_y&\textrm{if $x=yh'$,}\\
    Q_z&\textrm{if $x=zh'$.}
    \end{cases}
    \]
    See \cref{fig: augmenting1} for an illustration.
    \begin{figure}[tp]
    \centering  
    \includegraphics{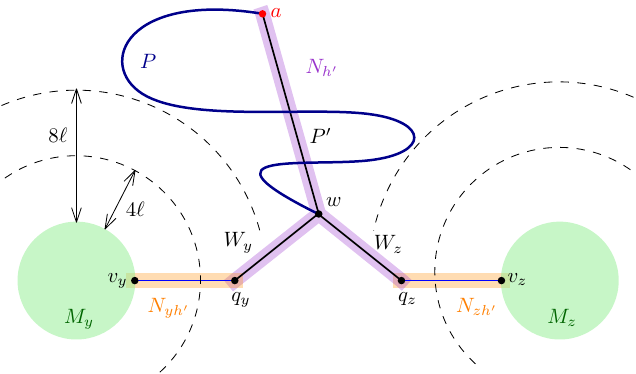}
    \caption{The model $\calN$ of $H'$ that we construct in Case 1 of the proof of~\cref{lem:augmenting-the-model}.}  
    \label{fig: augmenting1}
    \end{figure}

    We claim that $\calN = (N_x \mid x\in V(H')\cup E(H'))$ is a model of $H'$ in $G$, $\calN$ is $\ell$-fat, and $\calN$ satisfies \ref{item:modelH'-other-branchsets} and~\ref{item:modelH'-h'}. 
    Item~\ref{item:modelH'-other-branchsets} is satisfied by the first line of the definition of $\mathcal{N}$. 
    The first statement of~\ref{item:modelH'-h'} holds as $a\in V(P')\subseteq V(N_{h'})$.
    For the second statement, recall that $\length(W_y) = \length(W_z) = 4\ell$, and $\length(P') \leq 2\ell-1$.
    Since $w \in V(W_y) \cap V(W_z) \cap V(P')$, the radius of $N_{h'}$ is at most $4\ell$, as desired.
    Therefore, all we need to argue in this case is that $\calN$ is a model of $H'$ and $\calN$ is $\ell$-fat.

    Let $x$ and $x'$ be distinct elements of $V(H') \cup E(H')$.
    We consider all possible choices of 
    $x$ and $x'$ up to swapping them. 
    If $x$ and $x'$ are a vertex and an edge that are incident in $H'$, then we will show that $V(N_x) \cap V(N_{x'}) \neq \emptyset$.
    Otherwise, we will show that $\dist_G(V(N_x),V(N_{x'}))\geq\ell$.
    This will conclude the proof that $\calN$ is an $\ell$-fat model of $H'$ as $\ell > 0$.

    First, suppose that $x \in V(H')$, $x' \in E(H')$, and they are incident in $H'$.
    If $x\in V(H)$ and $x'\in E(H)$, then 
    $x$ and $x'$ are incident in $H$, 
    $N_x=M_x$, $N_{x'}=M_{x'}$, so since $\calM$ is a model of $H$, we have
    $V(N_x)\cap V(N_{x'}) = V(M_x)\cap V(M_{x'})\neq\emptyset$. 
    If $x\in \set{y,z}$ and $x'=xh'$, then 
    $N_x=M_x$ and $N_{x'}=Q_x$ so $v_x\in V(N_x) \cap V(N_{x'})$. 
    If $x=h'$ and $x' \in \{yh',zh'\}$, then $q_y,q_z \in V(N_x)$ and $\{q_y,q_z\} \cap V(N_{x'}) \neq \emptyset$, so $V(N_x) \cap V(N_{x'}) \neq \emptyset$.
    From now on, assume that $x$ and $x'$ are not a vertex and an edge that are incident in $H'$, and we prove that $\dist_G(V(N_x),V(N_{x'}))\geq\ell$.

    Next, we consider all the cases with $x \in V(H) \cup E(H) \setminus\{yz\}$. Fix such an $x$. 
    If $x' \in V(H) \cup E(H) \setminus\{yz\}$, then 
    \[
    \dist_G(V(N_x), V(N_{x'}))= \dist_G(V(M_x), V(M_{x'}))\geq 8\ell
    \]
    as $\calM$ is an $8\ell$-fat model.
    
    If $x \in V(H)$ and $x' = h'$, then 
    since $V(N_{h'})\subseteq B_G(w,4\ell)$, we have 
    \begin{align*}
    \dist_G(V(N_x),V(N_{h'})) 
    &\geq \dist_G(V(M_x),B_G(w,4\ell))\\
    &\geq \dist_G(V(M_x),w) - 4\ell\\
    &\geq 8\ell-4\ell=4\ell&&\textrm{by~\eqref{eq:P-far-from-the-model-vertices}.}
    \end{align*}
    If $x \in E(H) \setminus \{yz\}$ and $x' = h'$, then
    \begin{align*}
    \dist_G(V(N_x),V(P')) 
    &\geq \dist_G(V(M_x),B_G(w,2\ell-1))\\
    &\geq \dist_G(V(M_x),w)-2\ell\\
    &\geq 4\ell-2\ell = 2\ell&&\textrm{by~\eqref{eq:P-far-from-the-model-edges},}
    \end{align*} 
    and
    \begin{align*}
    \dist_G(V(N_x),V(W_y \cup W_z)) 
    &\geq \dist_G(V(M_x),B_G(V(M_{yz}),4\ell)\\
    &\geq \dist_G(V(M_x),V(M_{yz}))-4\ell\\
    &\geq 8\ell-4\ell = 4\ell&&\textrm{as $\calM$ is $\ell$-fat,}
    \end{align*}
    and therefore,
    \begin{align*}
    \dist_G(V(N_x),V(N_{h'})) 
    &= \max\{\dist_G(V(N_x),V(P')) , \dist_G(V(N_x),V(W_y \cup W_z)) \} \geq 2\ell.
    \end{align*} 
    If $x \notin \{y,z\}$ and $x' \in\set{yh',zh'}$, then since $\calM$ is $8\ell$-fat, we have
    \[\dist_G(V(N_x),V(N_{x'})) \geq \dist_G(V(M_x),V(Q_{y}\cup Q_z)) \geq \dist_G(V(M_x), V(M_{yz})) \geq 8\ell.
    \]
    If $(x,x')=(y,zh')$, then by~\eqref{eq:Q-far-from-M}, we have
    \[
    \dist_G(V(N_y),V(N_{zh'})) 
    = \dist_G(V(M_y), V(Q_z))
    \geq 4\ell.
    \]
    Symmetrically, for $(x,x')=(z,yh')$, we have
    \[
    \dist_G(V(N_z),V(N_{yh'})) 
    = \dist_G(V(M_z), V(Q_y))
    \geq 4\ell.
    \]
    This exhausts the cases with $x \in V(H) \cup E(H) \setminus\{y,z\}$.
    The only remaining case is $(x,x') = (yh', zh')$
    in which by the assumption of Case 1, we have
    \[
    \dist_G(V(N_{yh'}), V(N_{zh'})) = \dist_G(V(Q_y), V(Q_z)) \geq \ell.
    \]
    This concludes the proof that $\calN$ is an $\ell$-fat model of $H'$ in $G$ satisfying \ref{item:modelH'-other-branchsets} and~\ref{item:modelH'-h'}, and so the proof in Case 1 is completed.

    \textcolor{red}{Case 2.} $\dist_G(V(Q_y),V(Q_z)) \geq \ell$ and $\dist_G(a,w) \geq 2\ell$.

    For every $x\in V(H'')\cup E(H'')$, let
    \[
    N_x = \begin{cases}
    M_x&\textrm{if $x \in V(H)\cup E(H)\setminus\set{yz}$,}\\
    W_y\cup W_z&\textrm{if $x=h'$,}\\
    (\set{a},\emptyset)&\textrm{if $x=h''$,}\\
    Q_y&\textrm{if $x=yh'$,}\\
    Q_z&\textrm{if $x=zh'$,}\\
    P&\textrm{if $x=h'h''$.}
    \end{cases}
    \]
    See \cref{fig: augmenting2} for an illustration.
    We claim that $\calN = (N_x \mid x\in V(H'')\cup E(H''))$ is a model of $H''$ in $G$, $\calN$ is $\ell$-fat, and $\calN$ satisfies \ref{item:modelH''-other-branchsets}, \ref{item:modelH''-h'}, and~\ref{item:modelH''-h''}.
    Items~\ref{item:modelH''-other-branchsets} and~\ref{item:modelH''-h''} are satisfied by construction. 
    Item~\ref{item:modelH''-h'} is satisfied since $w \in V(W_y) \cap V(W_z)$ and $\length(W_y) = \length(W_z) = 4\ell$.
    Therefore, all we need to argue in this case is that $\calN$ is a model of $H''$ and $\calN$ is $\ell$-fat.

    \begin{figure}[tp]
    \centering  
    \includegraphics{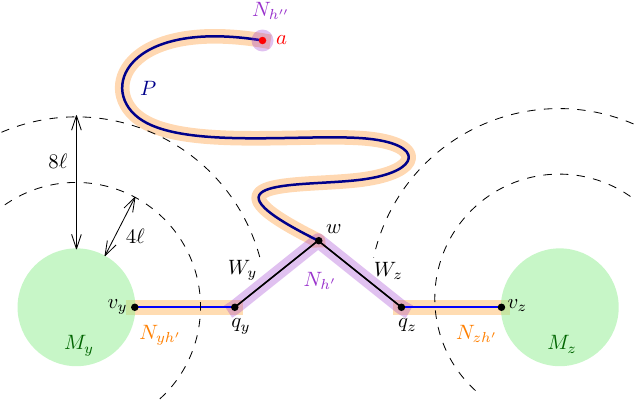}
    \caption{The model $\calN$ of $H''$ that we construct in Case 2 of the proof of~\cref{lem:augmenting-the-model}.
    }  
    \label{fig: augmenting2}
    \end{figure}

    \begin{figure}[tp]
    \centering  
    \includegraphics{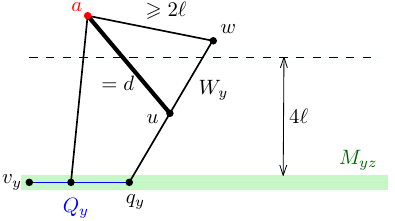}
    \caption{
    An illustration to the argument presented in Case 2 of~\Cref{lem:augmenting-the-model}, which shows that $\dist_G(V(N_{h'}, V(N_{h'h''})) > \ell$.
    }  
    \label{fig: augmenting-u-d}
    \end{figure}
    
    Let $x$ and $x'$ be distinct elements of $V(H'') \cup E(H'')$.
    We consider all possible choices of $x$ and $x'$ up to swapping them. 
    If $x$ and $x'$ are a vertex and an edge that are incident in $H''$, then we will show that $V(N_x) \cap V(N_{x'}) \neq \emptyset$.
    Otherwise, we will show that $\dist_G(V(N_x),V(N_{x'}))\geq\ell$.
    This will conclude the proof that $\calN$ is an $\ell$-fat model of $H''$ as $\ell > 0$.

    First, suppose that $x \in V(H'')$, $x' \in E(H'')$, and they are incident in $H''$.
    If $x\in V(H)$ and $x'\in E(H)$, then 
    $x$ and $x'$ are incident in $H$, 
    $N_x=M_x$, $N_{x'}=M_{x'}$, so since $\calM$ is a model of $H$, we have
    $V(N_x)\cap V(N_{x'}) = V(M_x)\cap V(M_{x'})\neq\emptyset$. 
    If $x\in \set{y,z}$ and $x'=xh'$, then 
    $N_x=M_x$ and $N_{x'}=Q_x$ so $v_x\in V(N_x) \cap V(N_{x'})$. 
    If $x=h'$ and $x' \in \{yh',zh',h'h''\}$, then $q_y,q_z,w \in V(N_x)$ and $\{q_y,q_z,w\} \cap V(N_{x'}) \neq \emptyset$, so $V(N_x) \cap V(N_{x'}) \neq \emptyset$.
    If $x = h''$ and $x' = h'h''$, then $a \in V(N_{x}) \cap V(N_{x'})$.
    From now on, assume that $x$ and $x'$ are not a vertex and an edge that are incident in $H''$, and we prove that $\dist_G(V(N_x),V(N_{x'}))\geq\ell$.

    Since $N_{h'}$, $N_{yh'}$, and $N_{zh'}$ are subgraphs of respective branch sets chosen in Case 1, 
    for all the pairs $(x,x')$ such that $x,x' \in (V(H'')\cup E(H''))\setminus\{h'', h'h''\} = V(H')\cup E(H')$ 
    the argument stays the same as in Case 1 
    (note that in Case 1, we use the assumption $\dist_G(a,w) \leq 2\ell-1$ only to prove~\ref{item:modelH'-h'}). 
    Thus, we only need to verify pairs $(x,x')$ with $x'\in\set{h'',h'h''}$.

    If $x\in V(H) \cup E(H)\setminus\{yz\}$ and $x' \in \{h'', h'h''\}$, 
    then since $N_{h''}\subseteq N_{h'h''} = P$ and by~\eqref{eq:P-far-from-the-model-vertices} and~\eqref{eq:P-far-from-the-model-edges}, we have 
    \[
    \dist_G(V(N_x), V(N_{x'})) \geq \dist_G(V(M_x),V(P)) \geq 4\ell.
    \]
    
    If $x \in \{yh', zh'\}$ and $x' \in \{h'', h'h''\}$, then by~\eqref{eq:aPw-far-Myz},
    \[\dist_G(V(N_x), V(N_{x'})) \geq \dist_G(V(Q_y \cup Q_z), V(P)) \geq 4\ell.\]
    Finally, we consider the case $(x,x') = (h', h'')$.
    Let $d = \dist_G(V(N_x), V(N_{x'})) =  \dist_G(V(W_y \cup W_z), a)$. 
    We shall prove that $d \geq \ell$.
    Without loss of generality, assume that $d = \dist_G(a,V(W_y))$ and let $u$ be a vertex of $W_y$ such that $\dist_G(a,u) = d$, see~\cref{fig: augmenting-u-d}.
    Recall that we assumed that $\dist_G(a,w) \geq 2\ell$. Thus, in particular, $a \neq w$ and 
    since $P$ is an $a$-$B_G(V(M_{yz}),4\ell)$ path in $G$, 
    it follows that $\dist_G(a,V(M_{yz})) \geq 4\ell+1$.
    We obtain
        \[d + \length(q_yW_yu) \geq 4\ell+1 \ \ \text{ and } \ \ d + \length(uW_yw) \geq 2\ell.\]
    Adding these two inequalities, we obtain $2d + \length(W_y) \geq 6\ell + 1$. Since $\length(W_y) = 4\ell$, we have $d > \ell$, as desired.
    This concludes the proof that $\calN$ is an $\ell$-fat model of $H''$ in $G$ satisfying \ref{item:modelH''-other-branchsets}, \ref{item:modelH''-h'}, and~\ref{item:modelH''-h''}, and thus ends Case 2.

    \textcolor{red}{Case 3.} $\dist_G(V(Q_y), V(Q_z)) < \ell$.

    Let $S$ be a $V(Q_y)$-$V(Q_z)$ path in $G$ satisfying $\length(S) = \dist_G(V(Q_y), V(Q_z))<\ell$.
    For each $x \in \{y,z\}$, let $s_x$ be the endpoint of $S$ in $Q_x$, and let $v_x'$ be the vertex of $M_{yz}$ at distance in $G$ exactly $3\ell$ from $v_x$.
    Since $\calM$ is $4\ell$-clean, $v_y'$ and $v_z'$ exist and are uniquely defined. 
    Also,
    \begin{equation}
    \label{eq:cut-M-yz-far-from-branchsets}
    \dist_G(V(v_y'M_{yz}v_z'), V(M_y\cup M_z)) \geq 3\ell.
    \end{equation}
    Now we claim that 
    \begin{equation}\label{eq:S-far-from-model}
        \dist_G(V(S),V(M_y\cup M_z))>3\ell.
    \end{equation}
    To show \eqref{eq:S-far-from-model}, we only prove that $\dist_G(V(S),V(M_y))> 3\ell$, as the inequality $\dist_G(V(S),V(M_z))> 3\ell$ has a symmetric argument. 
    For this, we let $u$ be a vertex on $S$ such that $d_G(u,V(M_y))=d_G(V(S),M_y)$. We then have 
    \begin{align*}
    \dist_G(V(S),V(M_y)) 
    & = \dist_G(u,V(M_y))\\
    &\geq \dist_G(s_z,V(M_y)) - \length(S)\\
    &> \dist_G(V(Q_z),V(M_y)) - \ell\\
    &\geq 4\ell - \ell = 3\ell &&\textrm{by~\eqref{eq:Q-far-from-M}.}
    \end{align*} 
    This proves \eqref{eq:S-far-from-model}. 
    We now let
    \[Q = v_y'Q_y q_y \cup v_z' Q_z q_z \cup S.\]
    See \cref{fig: augmenting3} for an illustration.
    \begin{figure}[tp]
    \centering  
    \includegraphics{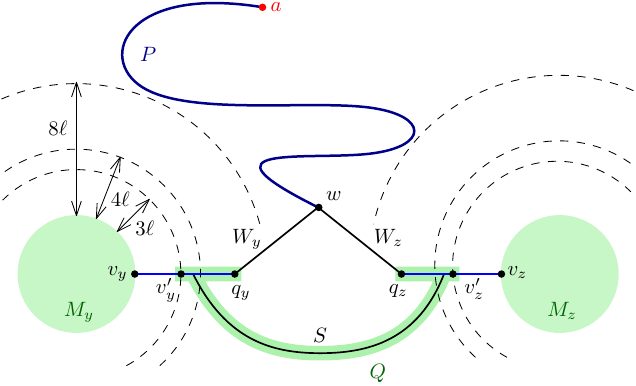}
    \caption{Setup for the application of the \nameref{lemma:tripod} in Case 3 of of~\Cref{lem:augmenting-the-model}.}  
    \label{fig: augmenting3}
    \end{figure}
    Note that \eqref{eq:S-far-from-model} implies that for each $x\in \{y,z\}$, the vertex $s_x$ belongs to $v'_xQ_xq_x$, thus $Q$ is connected.
    Furthermore, combining~\eqref{eq:cut-M-yz-far-from-branchsets} and~\eqref{eq:S-far-from-model} gives
    \begin{equation}\label{eq:far-from-Q}
        \dist_G(V(M_y \cup M_z), V(Q)) \geq 3\ell.
    \end{equation}
    Finally, note that since $\length(S) < \ell$ and by~\eqref{eq:aPw-far-Myz},
    \begin{equation}\label{eq:aPw-Q}
        \dist_G(V(P), V(Q)) \geq \dist_G(V(P), V(Q_y \cup Q_z)) - \length(S) > 4\ell-\ell = 3\ell.
    \end{equation}
    
    We plan to apply the \nameref{lemma:tripod} to $Q,v_y,v_z,w$ with $d = 4\ell$.
    To this end, we need to verify the assumptions \eqref{lem:tripod-asumpt-dist-xi-Q-large}, \eqref{lem:tripod-asumpt-dist-xi-Q-small}, and~\eqref{lem:tripod-asumpt-dist-xi-xj}.

    By~\eqref{eq:far-from-Q}, we have $\dist_G(\{v_y,v_z\},V(Q)) \geq 3\ell$, and by~\eqref{eq:aPw-Q}, we have $d_G(w, V(Q))>3\ell$.
    This proves assumption~\eqref{lem:tripod-asumpt-dist-xi-Q-large}.
    We have $\dist_G(v_y,v_y') = \dist_G(v_z,v_z') = 3\ell < 4\ell$ and $v_y', v_z' \in V(Q)$, hence $\dist_G(v_i,V(Q)) \leq 4\ell=d$ for each $i \in \{y,z\}$.
    Similarly, $\dist_G(w, q_y) = 4\ell$ and $q_y \in V(Q)$, hence $\dist_G(w, V(Q)) \leq 4\ell=d$.
    This proves assumption~\eqref{lem:tripod-asumpt-dist-xi-Q-small}.
    As $\calM$ is $8\ell$-fat, $v_y \in V(M_y)$, and $v_z \in V(M_z)$, we have $\dist_G(v_y,v_z) \geq 8 \ell=2d$.
    Since $w \in V(P)$, by~\eqref{eq:P-far-from-the-model-vertices}, $\dist_G(w, \{v_y,v_z\}) \geq 8\ell=2d$.
    This proves assumption~\eqref{lem:tripod-asumpt-dist-xi-xj}.

    Now we apply the \nameref{lemma:tripod} to $G$, $v_y,v_z,w$, $Q$, $\ell$, and $d = 4\ell$.
    We obtain connected subgraphs $Z$, $P_y$, $P_z$, $P_a$ in $G$ such that
    \begin{enumerate'}
    \item $v_i\in V(P_i)$ and $V(Z)\cap V(P_i)\neq\emptyset$ for each $i\in\{y,z,a\}$, 
    \label{lem:tripod-lemma:item:Pi-intersects-Z'}
    \item $Z \textrm{ has radius at most $\lfloor1.5\ell\rfloor$}$, 
    \label{lem:tripod-lemma:item:Z-has-bounded-radius'}
    \item $V(Z)\subseteq B_G(V(Q),2\ell-1)$,
    \label{lem:tripod-lemma:item:Z-is-close-to-Q'}
    \item $V(P_i)  \subseteq B_G(v_i,3\ell-1) \cup B_G(V(Q),\ell)$ for each $i\in\{y,z,a\}$,
    \label{lem:tripod-lemma:item:Pi-does-not-wander'}
    \item $\dist_G(V(P_i),V(P_j))\geq\ell$ for all distinct $i,j \in \{y,z,a\}$.
    \label{lem:tripod-lemma:item:Pi-Pj-are-far-apart'}
    \end{enumerate'}
    For every $x\in V(H'')\cup E(H'')$ let
    \[
    N_x = \begin{cases}
    M_x&\textrm{if $x \in V(H)\cup E(H)\setminus\set{yz}$,}\\
    Z&\textrm{if $x=h'$,}\\
    (\set{a},\emptyset)&\textrm{if $x=h''$,}\\
    P_y&\textrm{if $x=yh'$,}\\
    P_z&\textrm{if $x=zh'$,}\\
    P_a \cup P&\textrm{if $x=h'h''$.}
    \end{cases}
    \]
    See~\Cref{fig: augmenting3-model}.
    We state a few simple observations.
    Recall that $V(Q) \subset B_G(V(M_{yz}), \ell)$.
    Therefore, by~\ref{lem:tripod-lemma:item:Z-is-close-to-Q'},~\ref{lem:tripod-lemma:item:Pi-does-not-wander'}, and since $v_1,v_2\in V(M_{xy})$,
    \begin{equation*}
        V(Z) \cup V(P_1) \cup V(P_2) \subset B_G(V(M_{yz}), 3\ell-1).
    \end{equation*}
    It follows that
    \begin{equation}\label{eq:branch-sets-far-from-Myz}
        \textstyle\bigcup_{x \in \{h', yh', zh'\}} V(N_x) \subset B_G(V(M_{yz}), 3\ell-1).
    \end{equation}
    Additionally, by~\ref{lem:tripod-lemma:item:Pi-does-not-wander'},
    \begin{equation}\label{eq:branch-sets-far-from-Myz-and-P}
        V(N_{h'h''}) \subset B_G(V(P),3\ell-1) \cup B_G(V(Q), \ell).
    \end{equation}

    We claim that $\calN = (N_x \mid x\in V(H'')\cup E(H''))$ is a model of $H''$ in $G$, $\calN$ is $\ell$-fat, and $\calN$ satisfies \ref{item:modelH''-other-branchsets}, \ref{item:modelH''-h'}, and~\ref{item:modelH''-h''}. Items~\ref{item:modelH''-other-branchsets} and \ref{item:modelH''-h''} follow immediately from the definition of $\calN$, and 
    \ref{item:modelH''-h'} is an immediate consequence of the fact that $Z$ has radius at most $\lfloor1.5\ell\rfloor$.
    Therefore, all we need to argue in this case is that $\calN$ is a model of $H'$ and $\calN$ is $\ell$-fat.

    Let $x$ and $x'$ be distinct elements of $V(H'') \cup E(H'')$.
    We consider all possible choices of $x$ and $x'$ up to swapping them. 
    If $x$ and $x'$ are a vertex and an edge that are incident in $H''$, then we will show that $V(N_x) \cap V(N_{x'}) \neq \emptyset$.
    Otherwise, we will show that $\dist_G(V(N_x),V(N_{x'}))\geq\ell$.
    This will conclude the proof that $\calN$ is an $\ell$-fat model of $H''$ as $\ell > 0$.

    First, suppose that $x \in V(H'')$, $x' \in E(H'')$, and they are incident in $H''$.
    If $x\in V(H)$ and $x'\in E(H)$, then 
    $x$ and $x'$ are incident in $H$, 
    $N_x=M_x$, $N_{x'}=M_{x'}$, so since $\calM$ is a model of $H$, we have
    $V(N_x)\cap V(N_{x'}) = V(M_x)\cap V(M_{x'})\neq\emptyset$. 
    If $x\in \set{y,z}$ and $x'=xh'$, then 
    $N_x=M_x$ and $N_{x'}=P_x$ so $v_x\in V(N_x) \cap V(N_{x'})$ by~\ref{lem:tripod-lemma:item:Pi-intersects-Z'}. 
    If $x=h'$ and $x' \in \{yh',zh',h'h''\}$, then $V(N_x) \cap V(N_{x'}) \neq \emptyset$ by~\ref{lem:tripod-lemma:item:Pi-intersects-Z'}.
    If $(x,x') = (h'',h'h'')$, then since $a \in V(P)$, we have $V(N_x) \cap V(N_{x'}) \neq \emptyset$.
    From now on, assume that $x$ and $x'$ are not a vertex and an edge that are incident in $H'$, and we prove that $\dist_G(V(N_x),V(N_{x'}))\geq\ell$.

    \begin{figure}[tp]
    \centering  
    \includegraphics{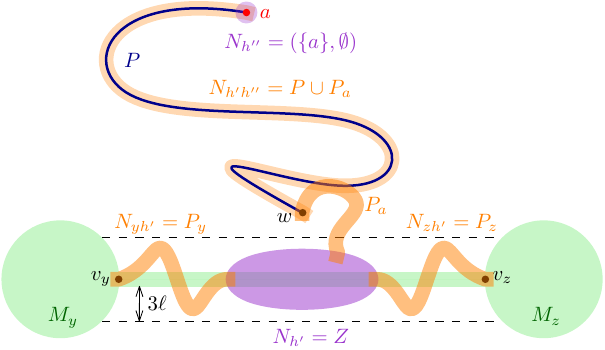}
    \caption{The model $\calN$ of $H''$ that we construct in Case 3 of the proof of~\Cref{lem:augmenting-the-model}.}  
    \label{fig: augmenting3-model}
    \end{figure}
    
    Suppose that $x \in V(H) \cup E(H) \setminus \{yz\}$.
    If $x' \in V(H) \cup E(H) \setminus \{yz\}$, then $\dist_G(V(N_x), V(N_{x'})) \geq 8\ell$ as $\calM$ is an $8\ell$-fat model.     

    If $x \notin \{y,z\}$ and $x' \in \{h', yh', zh'\}$, then 
    \begin{align*}
        \dist_G(V(N_x), V(N_{x'})) &\geq \dist_G(V(M_x), B_G(V(M_{yz}),3\ell-1)) && \text{by~\eqref{eq:branch-sets-far-from-Myz}}\\
        &\geq \dist_G(V(M_x), V(M_{yz})) - 3\ell+1    \\
        &\geq 8\ell - 3\ell = 5\ell && \text{as $\calM$ is $8\ell$-fat}.
    \end{align*}
    If $x \in \{y,z\}$ and $x' = h'$, then
    \begin{align*}
        \dist_G(V(N_x), V(N_{x'})) &=\dist_G(V(M_x), V(Z)) \\
        &\geq \dist_G(V(M_x),B_G(V(Q),2\ell-1)) && \text{ by~\ref{lem:tripod-lemma:item:Z-is-close-to-Q'}}\\
        &\geq \dist_G(V(M_x),V(Q)) - 2\ell\\
        &\geq 3\ell - 2\ell = \ell && \text{ by~\eqref{eq:far-from-Q}. }
    \end{align*}
    Next, consider the case of $x' = h'h''$. 
    We have,
    \begin{align*}
        \dist_G(V(N_x), V(N_{x'})) &= \dist_G(V(M_x), V(N_{h'h''})) \\
        &\geq \dist_G(V(M_x), B_G(V(P),3\ell-1)\cup B_G(V(Q),\ell)) && \text{by~\eqref{eq:branch-sets-far-from-Myz-and-P}}\\
        &\geq \min\{\dist_G(V(M_x), V(P)) - 3\ell, \dist_G(V(M_x), V(Q)) - \ell\}.
    \end{align*}
    By~\eqref{eq:P-far-from-the-model-vertices} and~\eqref{eq:P-far-from-the-model-edges},
    $\dist_G(V(M_x), V(P)) \geq 4\ell$.
    If $x \in \{y,z\}$, then 
    \[\dist_G(V(M_x), V(Q)) \geq \dist_G(V(M_x), V(M_{yz})) - \length(S) \geq  8\ell-\ell = 7\ell.\] 
    If $x \in \{y,z\}$, then by~\eqref{eq:far-from-Q},
    $\dist_G(V(M_x), V(Q)) \geq 3\ell$.
    Altogether, we obtain $\dist_G(V(N_x), V(N_{x'})) \geq 2\ell$.
    As $N_{h''}\subseteq N_{h'h''}$, note that when $x' = h''$, we also obtain $\dist_G(V(N_x), V(N_{x'})) \geq 2\ell$.
    
    Finally, consider the case where $(x,x') = (y,zh')$, and observe that the case where $(x,x') = (z,yh')$ is completely symmetric. In this case, we have
    \begin{align*}
        \dist_G(V(N_x), V(N_{x'})) &\geq \dist_G(V(M_y), V(P_z)) \\
        &\geq \dist_G(V(M_y), B_G(V(M_z),3\ell-1)\cup B_G(V(Q),\ell)) &&\text{by~\ref{lem:tripod-lemma:item:Pi-does-not-wander'}}\\
        &\geq \min \{\dist_G(V(M_y), V(M_z))-3\ell+1,\\
        &\hphantom{\geq \min \{}\dist_G(V(M_y), V(Q)) - \ell\}\\
        &\geq \min\{8\ell-3\ell+1, 3\ell - \ell\} = 2\ell 
    \end{align*}
    where the last inequality follows from the assumption that $\calM$ is $8\ell$-fat and by~\eqref{eq:far-from-Q}. This completes the proof in all cases where $x \in V(H) \cup E(H) \setminus \{yz\}$.

    If $(x,x') = (yh', zh')$, then $\dist_G(V(N_x), V(N_{x'})) = \dist_G(V(P_y), V(P_z)) \geq \ell$ by~\ref{lem:tripod-lemma:item:Pi-Pj-are-far-apart'}.
    Next, suppose that $x = h'h''$ and $x' \in \{yh',zh'\}$.
    Say that $x' = yh'$ (the argument for $x' = zh'$ is symmetric).
    We have,
    \begin{align*}
        \dist_G(V(N_x), V(N_{x'})) &=\dist_G(V(P_a), V(P_y))\\
        &\geq \min\{\dist_G(V(P_3), V(P_y)), \\
        &\hphantom{\geq \min \{}\dist_G(V(P),V(P_y))\}\\
        &\geq \min\{\ell, \dist_G(V(P),V(P_y))\} &&\text{by~\ref{lem:tripod-lemma:item:Pi-Pj-are-far-apart'}}\\
        &\geq \min \{\ell, \dist_G(V(P), V(M_y))-3\ell+1, \\
        &\hphantom{\geq \min \{\ell,}\dist_G(V(P), V(Q))-\ell\} &&\text{by~\ref{lem:tripod-lemma:item:Pi-does-not-wander'}}\\
        &\geq \min \{\ell, 4\ell-3\ell+1, 3\ell-\ell\} = \ell &&\text{by~\eqref{eq:P-far-from-the-model-vertices}, \eqref{eq:P-far-from-the-model-edges}, and \eqref{eq:aPw-Q}.} 
    \end{align*}
    Since, $V(N_{h''}) \subset V(N_{h'h''})$, we also obtain $\dist_G(V(N_x), V(N_{x'})) \geq \ell$ in the case where $x = h''$ and $x' \in \{yh',zh'\}$.
    Finally, suppose that $(x,x') = (h',h'')$.
    We have,
    \begin{align*}
        \dist_G(V(N_x), V(N_{x'})) &=\dist_G(V(Z), a)\\
        &\geq \dist_G(V(Z), V(P)) \\
        &\geq \dist_G(V(Q),V(P)) - 2\ell+1 &&\text{by~\ref{lem:tripod-lemma:item:Z-is-close-to-Q'}}\\
        &\geq 3\ell - 2\ell+1 = \ell+1 &&\text{by~\eqref{eq:aPw-Q}.}
    \end{align*}
    This concludes the proof that $\calN$ is an $\ell$-fat model of $H''$ in $G$ satisfying \ref{item:modelH''-other-branchsets}, \ref{item:modelH''-h'}, and~\ref{item:modelH''-h''}, and thus finishes Case 3.
    Since Cases 1, 2, and 3 are complementary, the proof is completed.
\end{proof}

\section{Observations on subcubic forests}
\label{sec:forests}

In this section, we prove several basic statements on subcubic forests. 
These will be used in the next section to find a collection of vertex-disjoint paths within a frame (i.e.\ within a subcubic forest), provided that it is large enough. 

\begin{lemma}\label{lem:bound-branch-vertices}
    Let $T$ be a subcubic tree with at least two vertices 
    and for each $\alpha \in[3]$, let $V_\alpha = \{v \in V(T) \mid \degree_T(v) = \alpha\}$.
    Then, $|V_3| \leq |V_1| - 2$.
\end{lemma}
\begin{proof}
    The proof is by induction on $|V(T)|$. 
    If $T$ has two vertices, then $|V_3|=0$ and $|V_1|=2$, 
    so the lemma statement holds.
    Now suppose that $|V(T)|\geq3$ and let $v$ be a vertex of degree $1$ (a leaf) in $T$. 
    Let $u$ be the unique neighbor of $v$ in $T$. 
    Since $T$ is a subcubic tree and $T$ has more than two vertices, the degree of $u$ in $T$ is either $2$ or $3$. 
    Let $T'=T-v$ and let $V_\alpha'= \set{v \in V(T') \mid \degree_{T'}(v) = \alpha}$, for each $\alpha\in[3]$. 
    The induction hypothesis for $T'$ gives $|V_3'|\leq |V_1'|-2$.
    If the degree of $u$ in $T$ is $2$, then $|V_1|=|V_1'|$ and $|V_3|=|V_3'|$, so we get
    $|V_3| = |V_3'| \leq |V_1'|-2 = |V_1|-2$.
    If the degree of $u$ is $3$, then $|V_1|=|V_1'|+1$ and $|V_3|=|V_3'|+1$, so we get 
    $|V_3| = |V_3'|+1 \leq |V_1'|+1-2 = |V_1|-2$.
    This completes the proof.
\end{proof}

\begin{corollary}\label{cor:bound-branch-vertices}
    Let $F$ be a subcubic forest where each of $m$ components has at least two vertices
    and for each $\alpha \in[3]$, let $V_\alpha = \{v \in V(F) \mid \degree_F(v) = \alpha\}$.
    Then, $|V_3| \leq |V_1| - 2m$.\end{corollary}

\begin{lemma}
\label{lemma:extracting-paths-from-the-frame}
Let $T$ be a subcubic tree and 
let $Z\subseteq \set{v\in V(T)\mid \degree_T(v)\leq 2}$. 
Then $T$ contains $\lfloor\frac{|Z|}{2}\rfloor$ pairwise vertex-disjoint $Z$-paths.
\end{lemma}

\begin{proof}
The proof is by induction on $|V(T)|$. 
If $T$ contains at most one vertex, then there is nothing to prove. 
If $T$ contains a vertex $v$ of degree $1$ (a leaf) which is not in $Z$, 
then we call induction for $T-v$ and $Z$. 
Thus, from now on, we assume that $T$ has at least two vertices and all the leaves of $T$ are in $Z$.

Next, assume that $T$ contains a vertex $v$ of degree $2$ with $v \notin Z$.
Consider a tree $T'$, which is obtained from $T-v$ by adding an edge $uw$ between the neighbors $u$ and $w$ of $v$ in $T$.
By induction applied to $T'$ and $Z$, we obtain $\lfloor\frac{|Z|}{2}\rfloor$ pairwise vertex-disjoint $Z$-paths in $T'$.
If one of the paths contains $uw$, then we replace it by $uvw$ and obtain $\lfloor\frac{|Z|}{2}\rfloor$ pairwise vertex-disjoint $Z$-paths in $T$.
Thus, from now on, we may assume that all vertices of degree at most $2$ in $T$ are also in $Z$.

Assume that there exist two vertices $u$ and $v$ of $T$ such that $uv\in E(T)$, $u$ is a leaf of $T$, and $\deg_T(v) \leq 2$. 
Thus, $u,v\in Z$. 
Then $T'=T-\{u,v\}$ is a subcubic tree, and induction hypothesis implies that it contains a collection of $\lfloor \tfrac{|Z\cap V(T')|}{2}\rfloor=\lfloor \tfrac{|V(Z)|}{2}\rfloor-1$ pairwise vertex-disjoint $Z$-paths. 
Adding the $Z$-path $uv$ to the collection we obtain the lemma statement.

Finally, we assume that the previous case does not hold.
It follows that $T$ contains a vertex of degree $3$.
Let $r$ be an arbitrary vertex of $T$ and let $u$ be a vertex of degree $3$ in $T$ that maximises $\dist_T(r,u)$. 
Thus, $u$ has two neighbors in $T$ not contained in the $r$-$u$ path in  $T$, say $v$ and $w$. 
Since we are not in the previous case and by choice of $u$, 
it follows that $v$ and $w$ are leafs in $T$.
Let $T' = T - \{u,v,w\}$ and $Z' = Z \setminus \{v,w\}$. Note that $T'$ is a subcubic forest. 
By induction applied to $T'$ and $Z'$, we obtain $\lfloor\frac{|Z|-2}{2}\rfloor=\lfloor\frac{|Z|}{2}\rfloor-1$ pairwise vertex-disjoint $Z$-paths in $T$.
By adding the $Z$-path $vuw$ to the collection, we conclude the proof.
\end{proof}

\begin{corollary}
\label{cor:extracting-paths-from-the-frame}
Let $F$ be a subcubic forest and let $m$ be the number of components of $F$.
Let $Z\subseteq \set{v\in V(F)\mid \degree_F(v)\leq 2}$. 
Then $F$ contains $\left\lceil\frac{|Z|-m}{2}\right\rceil$ pairwise vertex-disjoint $Z$-paths.
\end{corollary}
\begin{proof}
    Let $T_1,\dots,T_m$ be the components of $F$.
    Applying~\Cref{lemma:extracting-paths-from-the-frame} to each of the components, we obtain a collection of pairwise disjoint $Z$-paths of size at least
        \[\sum_{i=1}^m \left\lfloor\frac{|Z \cap V(T_i)|}{2}\right\rfloor = \sum_{i=1}^m \left\lceil\frac{|Z \cap V(T_i)| - 1}{2}\right\rceil \geq \left\lceil\sum_{i=1}^m \frac{|Z \cap V(T_i)| - 1}{2}\right\rceil = \left\lceil\frac{|Z|-m}{2}\right\rceil. \qedhere\]
\end{proof}

\section{Wrapping up}\label{sec:wrapping}

In this section, we complete the proofs of~\cref{theorem:main,theorem:main2}.
First, we formally define frames which we will work with.
\Cref{lem:larger-frame-or-hitting} below encapsulates the induction step of the final proof: given a frame, we either find a larger frame or a hitting set.
In~\Cref{lem:large-frame-implies-packing}, we prove that a large enough frame contains a required packing.

Let $i$ be a nonnegative integer, let $\ell$ and $r$ be positive integers, let $G$ be a graph, and let $A$ be a subset of the vertices of $G$.
A pair $(F,\calM)$ is an \defin{$(i,\ell,r,A)$-frame} in $G$ if $F$ is a subcubic forest and $\calM = (M_x \mid x \in V(F) \cup E(F))$ is a model of $F$ in $G$ satisfying the following conditions.
Let $m$ be the number of components of $F$ containing at least two vertices, and for each $\alpha \in \{0,1,2,3\}$, let
\begin{align*}
    V_\alpha &= \{v \in V(F) \mid \deg_F(v) = \alpha\}.
\end{align*}
Then we have
\begin{enumerate}[label=\textup{(f\arabic*)}, noitemsep]
    \item $i = |V_0| + |V_1| + |V_2| - m$, \label{frame:i}
    \item $\calM$ is $\ell$-fat, \label{frame:ell}
    \item $M_x$ has radius at most $r$ for every $x \in V(F)$, \label{frame:radius}
    \item $A \cap V(M_x) \neq \emptyset$ for every $x \in V_1 \cup V_2$, \label{frame:Z}
    \item $M_x$ is an $A$-path for every $x \in V_0$. \label{frame:Y}
\end{enumerate}

Additionally, an $(i,\ell,r,A)$-frame $(F,\calM)$ is \defin{coarse} if all vertices of $F$ have positive degree, i.e.\ $V_0 = \{v \in V(F) \mid \deg_F(v) = 0\} = \emptyset$. 

\begin{lemma}\label{lem:larger-frame-or-hitting}
    Let $i$ be a nonnegative integer, $\ell$ and $r$ be positive integers with $r \geq 4\ell$, let $G$ be a graph, and let $A$ be a subset of the vertices of $G$.
    Let $(F, \calM')$ be an $(i,16\ell,r,A)$-frame in $G$. 
    Then either
    \begin{enumerate}
        \item $G$ has an $(i+1,\ell,r,A)$-frame, or \label{normal:new-frame}
        \item there exists $X \subset V(G)$ with $|X| \leq 2i$ such that every $A$-path in $G$ contains a vertex in $B_G(X, r + 8\ell)$. \label{normal:hitting}
    \end{enumerate}
    Additionally, if $(F,\calM')$ is coarse, then either
    \begin{enumerate}[label=\textup{(\roman*-c)}, noitemsep]
        \item $G$ has a coarse $(i+1,\ell,r,A)$-frame in $G$, or \label{corase:new-frame}
        \item there exists $X \subset V(G)$ with $|X| \leq 2i$ such that every $\ell$-coarse $A$-path in $G$ contains a vertex in $B_G(X, r + 8\ell)$. \label{coarse:hitting}
    \end{enumerate}
\end{lemma}
\begin{proof}
    Let $\calM' = (M_x' \mid x \in V(F) \cup E(F))$, let $V_\alpha = \{v \in V(F) \mid \deg_F(v) = \alpha\}$ for each $\alpha \in \{0,1,2,3\}$, and let $m$ denote the number of components of $F$ with at least two vertices.
    First, we apply~\Cref{lemma:fat-to-clean} to $G$ and $\calM'$, and obtain a model 
    $\calM = (M_x \mid x \in V(F) \cup E(F))$ of $F$ in $G$, which is $8\ell$-fat, $4\ell$-clean, and such that for every $x \in V(F)$, we have $M_x = M_x'$.
    In particular, $(F,\calM)$ is an $(i,8\ell,r,A)$-frame in $G$ and if moreover $(F,\calM')$ is coarse, then so is $(F,\calM)$.

    For each $x \in V(F)$, let $c_x$ be a vertex of $G$ such that $V(M_x) \subset B_G(c_x,r)$ (such a vertex exists by~\ref{frame:radius}). 
    We define $X = \{c_x \mid x \in V(F)\}$.
    In particular, $\bigcup_{x \in V(F)} V(M_x) \subset  B_G(X, r)$ and $B_G(\bigcup_{x \in V(F)} V(M_x), 8\ell) \subset B_G(X, r+8\ell)$.
    Additionally,
    \begin{align*}
        |X| = |V(F)| &= |V_0| + |V_1| + |V_2| + |V_3|\\
        &\leq |V_0| + |V_1| + |V_2| + (|V_1| - 2m) && \text{by~\Cref{cor:bound-branch-vertices}}\\
        &\leq 2(|V_0| + |V_1| + |V_2| - m)\\
        &= 2i &&\text{by~\ref{frame:i}.}
    \end{align*}
    Therefore, if every $A$-path in $G$ contains a vertex in $B_G(X, r + 8\ell)$, then~\ref{normal:hitting} holds, and if every $\ell$-coarse $A$-path in $G$ contains a vertex in $B_G(X, r + 8\ell)$, then~\ref{coarse:hitting} holds.
    Thus, from now on, we assume that there is an $A$-path $P$ that is disjoint from $B_G(X, r + 8\ell)$, and so 
    \begin{equation}\label{eq:P-far-from-the-vertices-model}
        \dist_G\left(V(P),\bigcup_{x \in V(F)} V(M_x)\right) \geq 8\ell.
    \end{equation}
    Moreover, we choose $P$ to be $\ell$-coarse if possible.
    Let $a$ and $a'$ be the endpoints of $P$. 
    Note that $a,a'\in A$. 
    We split the reasoning depending on 
    the distance from $P$ to the branch paths, and the distance between $a$ and $a'$. 
    
    \textcolor{red}{Case 1.} $\dist_G(V(P), \bigcup_{x \in E(F)}V(M_x)) \leq 4\ell$.

    Let $w$ be the first vertex in $P$ starting from $a$ with $\dist_G(w,\bigcup_{x \in E(F)}V(M_x)) \leq 4\ell$ and let $P' = aPw$.
    Then there exists $yz \in E(F)$ such that $\dist_G(w, V(M_{yz})) \leq 4\ell$, and since $\calM$ is $8\ell$-fat (by~\ref{frame:ell}),
    \begin{equation}\label{eq:P-far-from-edges}
        \dist_G\left(V(P'), \bigcup_{x\in E(H) \setminus\{yz\}} V(M_x)\right) \geq 4\ell.
    \end{equation}
    
    We apply~\Cref{lem:augmenting-the-model} to $\ell$, $G$, $H = F$, $\calM$, $a$, $yz$, and $P'$ (note that~\eqref{eq:P-far-from-the-model-vertices} holds by~\eqref{eq:P-far-from-the-vertices-model} and~\eqref{eq:P-far-from-the-model-edges} holds by~\eqref{eq:P-far-from-edges}).
    Let $H'$, $H''$, $h'$, and $h''$ be as in the assertion of~\Cref{lem:augmenting-the-model}.
    This assertion has two cases.
    
    First, assume that we obtained an $\ell$-fat model $\calN = (N_x \mid x \in V(H') \cup E(H'))$ of $H'$ in $G$ satisfying~\ref{item:modelH'-other-branchsets} and~\ref{item:modelH'-h'}.
    We claim that $(H', \calN)$ is an $(i+1,\ell,r,A)$-frame in $G$.
    By construction, $H'$ is a subcubic forest and $\calN$ is a model of $H'$ in $G$.
    For each $\alpha \in \{0,1,2,3\}$ let $V_\alpha' = \{v \in V(H') \mid \deg_{H'}(v) = \alpha\}$ and let $m'$ be the number of components of $H'$ with at least two vertices.
    
    We have $V_0' = V_0$, $V_1' = V_1$, $V_2' = V_2 \cup \{h'\}$, and $m' = m$, hence by~\ref{frame:i} for $(F,\calM)$, we have $i+1 = |V_0'| + |V_1'| + |V_2'| - m'$, and so~\ref{frame:i} holds for $(H',\calN)$.
    As $\calN$ is $\ell$-fat,~\ref{frame:ell} holds for $(H',\calN)$.
    Item~\ref{item:modelH'-other-branchsets} implies that $M_x = N_x$ for every $x \in V(F)$ and \ref{item:modelH'-h'} implies that $A \cap V(N_{h'})) \neq \emptyset$ (as $a \in A$) and that $N_{h'}$ has radius at most $4\ell$.
    By assumption, $4\ell \leq r$. We also have $V_0' = V_0$.
    It follows that by \ref{frame:radius}, \ref{frame:Z}, and \ref{frame:Y} for $(F,\calM)$, items \ref{frame:radius}, \ref{frame:Z}, and \ref{frame:Y} hold for $(H',\calN)$.
    Additionally, if $(F,\calM)$ is coarse (i.e.\ $V_0 = \emptyset$), then $(H',\calN)$ is coarse (i.e.\ $V_0' = \emptyset$).

    Second, assume that we obtained an $\ell$-fat model $\calN = (N_x \mid x \in V(H') \cup E(H''))$ of $H''$ in $G$ satisfying~\ref{item:modelH''-other-branchsets}, \ref{item:modelH''-h'}, and~\ref{item:modelH''-h''}.
    We claim that $(H'',\calN)$ is an $(i+1,\ell,r,A)$-frame in $G$.
    By construction, $H''$ is a subcubic forest and $\calN$ is a model of $H''$ in $G$.
    For each $\alpha \in \{0,1,2,3\}$ let $V_\alpha'' = \{v \in V(H'') \mid \deg_{H''}(v) = \alpha\}$ and let $m''$ be the number of components of $H''$ with at least two vertices.

    We have $V_0'' = V_0$, $V_1'' = V_1 \cup \{h''\}$, $V_2'' = V_2$, and $m'' = m$, hence by~\ref{frame:i} for $(F,\calM)$, we have $i+1 = |V_0''| + |V_1''| + |V_2''| - m''$, and so~\ref{frame:i} holds for $(H'',\calN)$.
    As $\calN$ is $\ell$-fat,~\ref{frame:ell} holds for $(H'',\calN)$.
    Item~\ref{item:modelH''-other-branchsets} implies that $M_x = N_x$ for every $x \in V(F)$, \ref{item:modelH''-h'} states that $N_{h'}$ has radius at most $4\ell$, and \ref{item:modelH''-h''} implies that $A \cap V(N_{h''})) \neq \emptyset$ (as $a \in A$) and that the radius of $N_{h''}$ is $0$. By assumption, $4\ell \leq r$. We also have $V_0'' = V_0$.
    It follows that by \ref{frame:radius}, \ref{frame:Z}, and \ref{frame:Y} for $(F,\calM)$, items \ref{frame:radius}, \ref{frame:Z}, and \ref{frame:Y} hold for $(H'',\calN)$.
    Additionally, if $(F,\calM)$ is coarse (i.e.\ $V_0 = \emptyset$), then $(H'',\calN)$ is coarse (i.e.\ $V_0'' = \emptyset$).

    We proved that $(H',\calN)$ or $(H'',\calN)$ (depending on the outcome of the application of~\cref{lem:augmenting-the-model}) is an $(i+1,\ell,r,A)$-frame in $G$, hence~\ref{normal:new-frame} holds.
    Moreover, if $(F,\calM)$ is coarse, then $(H',\calN)$ or $(H'',\calN)$ is also coarse, hence~\ref{corase:new-frame} holds.

    \textcolor{red}{Case 2.} $\dist_G(V(P), \bigcup_{x \in E(F)}V(M_x)) \geq 4\ell$ and $\dist_G(a,a') \geq \ell$.

    Let $F'$ be obtained from $F$ by adding two new vertices $h$ and $h'$ adjacent only to each other in $F'$.
    For every $x\in V(F')\cup E(F')$ let
    \[
    N_x = \begin{cases}
    M_x&\textrm{if $x \in V(F)\cup E(F)$,}\\
    (\set{a},\emptyset)&\textrm{if $x=h$,}\\
    (\set{a'},\emptyset)&\textrm{if $x=h'$,}\\
    P&\textrm{if $x=hh'$.}
    \end{cases}
    \]  
    We verify that $\calN = (N_x \mid x\in V(F')\cup E(F'))$ is a model of $F'$ in $G$ and $\calN$ is $\ell$-fat.
    Let $x$ and $x'$ be distinct elements of $V(F') \cup E(F')$.
    We consider all possible choices of $x$ and $x'$ up to swapping them. 
    By construction of $\mathcal{N}$ and since $\calM$ is a model of $F$, if $x$ and $x'$ are a vertex and an edge that are incident in $F'$, then $V(N_x) \cap V(N_{x'}) \neq \emptyset$.
    Thus, assume otherwise.
    We will show that $\dist_G(V(N_x),V(N_{x'}))\geq\ell$.
    This will conclude the proof that $\calN$ is an $\ell$-fat model of $F'$ as $\ell > 0$.
    If $x,x' \in V(F) \cup E(F)$, then $\dist_G(V(N_x),V(N_{x'})) = \dist_G(V(M_x),V(M_{x'})) \geq 8\ell$ as $\calM$ is $8\ell$-fat (by~\ref{frame:ell} for $(F,\calM)$).
    If $x \in V(F) \cup E(F)$ and $x' \in \{h,hh',h'\}$, then $V(N_x) \subset V(P)$, and so by~\eqref{eq:P-far-from-the-vertices-model} and 
    the Case 2 assumption, we have $\dist_G(V(N_x),V(N_{x'})) \geq \dist_G(V(M_x),V(P)) \geq 4\ell$.
    If $x,x' \in \{h,hh',h'\}$, then $\{x,x'\} = \{h,h'\}$ and $\dist_G(V(N_x), V(N_{x'})) = \dist_G(a,a') \geq \ell$ by the Case 2 assumption.
    This concludes the proof that $\calN$ is an $\ell$-fat model of $F'$ in $G$.
    
    We claim that $(F',\calN)$ is an $(i+1, \ell, r, A)$-frame in $G$.
    We have already proved that $\calN$ is $\ell$-fat, and so~\ref{frame:ell} holds for $(F',\calN)$.
    
    For each $\alpha \in \{0,1,2,3\}$ let $V_\alpha' = \{v \in V(F') \mid \deg_{F'}(v) = \alpha\}$ and let $m'$ be the number of components of $F'$ with at least two vertices.
    We have $|V_0'| = |V_0|$, $|V_1'| = |V_1| + 2$, $|V_2'| = |V_2|$, and $m' = m + 1$, hence by~\ref{frame:i} for $(F,\calM)$, we have $i+1 = |V_0'| + |V_1'| + |V_2'| - m'$, and so~\ref{frame:i} holds for $(F',\calN)$.
    
    The radius of $N_{h}$ and $N_{h'}$ is equal to $0$ and both sets $V(N_{h})$ and $V(N_{h'})$ contain a vertex in $A$.
    Therefore, by \ref{frame:radius} and \ref{frame:Z} for $(F,\calM)$, we obtain \ref{frame:radius} and \ref{frame:Z} for $(F',\calN)$.
    Finally, \ref{frame:Y} for $(F',\calN)$ follows from \ref{frame:Y} for $(F,\calM)$ and as $V_0 = V_0'$.
    Additionally, if $(F,\calM)$ is coarse, then $(F',\calN)$ is also coarse.

    This shows that $(F',\calN)$ is an $(i+1, \ell, r, A)$-frame in $G$ as claimed, and so~\ref{normal:new-frame} holds.
    Moreover, if $(F,\calM)$ is coarse, then $(F',\calN)$ is also coarse, hence~\ref{corase:new-frame} holds.

    \textcolor{red}{Case 3.} $\dist_G(V(P), \bigcup_{x \in E(F)}V(M_x)) \geq 4\ell$ and $\dist_G(a,a') < \ell$.

    Note that in this case, $P$ is not $\ell$-coarse. 
    Thus, by our construction rule, there is no $\ell$-coarse $A$-path disjoint from $B_G(X,r+8\ell)$ in $G$. It follows that \ref{coarse:hitting} holds.

    We now show that~\ref{normal:new-frame} holds.
    Let $F'$ be obtained from $F$ by adding a new isolated vertex $h$.
    Let $P'$ be a shortest $a$-$a'$ path in $G$.
    For every $x\in V(F')\cup E(F')$, let
    \[
    N_x = \begin{cases}
    M_x&\textrm{if $x \in V(F)\cup E(F)$,}\\
    P'&\textrm{if $x=h$,}
    \end{cases}
    \]
    We verify that $\calN = (N_x \mid x\in V(F')\cup E(F'))$ is a model of $F'$ in $G$ and $\calN$ is $\ell$-fat.
    Let $x$ and $x'$ be distinct elements of $V(F') \cup E(F')$.
    We consider all possible choices of $x$ and $x'$ up to swapping them. 
    By construction of $\mathcal{N}$ and since $\calM$ is a model of $F$, if $x$ and $x'$ are a vertex and an edge that are incident in $F'$, then $V(N_x) \cap V(N_{x'}) \neq \emptyset$.
    Thus, assume otherwise.
    We will show that $\dist_G(V(N_x),V(N_{x'}))\geq\ell$.
    This will conclude the proof that $\calN$ is an $\ell$-fat model of $F'$ as $\ell > 0$.
    If $x,x' \in V(F) \cup E(F)$, then $\dist_G(V(N_x),V(N_{x'})) = \dist_G(V(M_x),V(M_{x'})) \geq 8\ell$ as $\calM$ is $8\ell$-fat.
    If $x \in V(F) \cup E(F)$ and $x' = h$, then,
    \begin{align*}
        \dist_G(V(N_x),V(N_{x'})) &= \dist_G(V(M_x),V(P'))\\
        &\geq \dist_G(V(M_x), V(P)) - \length(P') \\
        &> 4\ell - \ell = 3\ell && \text{by~\eqref{eq:P-far-from-the-vertices-model} and the Case 3 assumption.}
    \end{align*}
    This concludes the proof that $\calN$ is an $\ell$-fat model of $F'$ in $G$.
    
    We claim that $(F',\calN)$ is an $(i+1, \ell, r, A)$-frame in $G$.
    We have already proved that $\calN$ is $\ell$-fat, and so~\ref{frame:ell} holds for $(F',\calN)$.

    For each $\alpha \in \{0,1,2,3\}$, let $V_\alpha' = \{v \in V(F') \mid \deg_{F'}(v) = \alpha\}$ and let $m'$ denote the number of components of $F'$ with at least two vertices.
    We have $|V_0'| = |V_0|+1$, $|V_1'| = |V_1|$, $|V_2'| = |V_2|$, and $m' = m$, hence by~\ref{frame:i} for $(F,\calM)$, we have $i+1 = |V_0'| + |V_1'| + |V_2'| - m'$, and so~\ref{frame:i} holds for $(F',\calN)$.
    
    The radius of $N_h$ is less than $\ell$, $N_h$ contains (at least two) vertices of $A$, $N_h$ is an $A$-path, and $V_0' = V_0 \cup \{h\}$.
    Therefore, by \ref{frame:radius}, \ref{frame:Z}, and~\ref{frame:Y} for $(F,\calM)$, we obtain \ref{frame:radius}, \ref{frame:Z}, and~\ref{frame:Y} for $(F',\calN)$.

    This shows that $(F',\calN)$ is an $(i+1, \ell, r, A)$-frame in $G$ as claimed and so~\ref{normal:new-frame} holds.

    Since Cases 1, 2, and 3 are complementary, the proof is complete.
\end{proof}

We will use the following straightforward observation.

\begin{obs}\label{obs:packing}
    Let $\ell$ be a positive integer, let $G$ and $H$ be graphs, and let $A$ be a subset of the vertices of $G$.
    Let $\calM$ be an $\ell$-fat model of $H$ in $G$.
    Let $Q_1, \dots, Q_j$ be pairwise disjoint subgraphs of $H$ and for each $\beta \in [j]$, let $P_\beta$ be a subgraph of 
    $\bigcup_{x \in V(Q_\beta)\cup E(Q_\beta)} M_x$.
    Then, $P_1,\dots,P_j$ is a collection of subgraphs of $G$, which are pairwise at distance at least $\ell$ in $G$.
\end{obs}

\begin{lemma} \label{lem:large-frame-implies-packing}
    Let $i$, $\ell$, and $r$ be positive integers, let $G$ be a graph, and let $A$ be a subset of the vertices of $G$.
    Let $(F,\calM)$ be a $(2i-1,\ell,r,A)$-frame in $G$.
    Then there is a family $\mathcal{P}$ of $A$-paths in $G$ with 
    $|\mathcal{P}|=i$ such that the paths in $\calP$ are pairwise at distance at least $\ell$ in $G$.
    Additionally, if $(F,\calM)$ is coarse, then there is a family $\mathcal{P}$ of $\ell$-coarse $A$-paths in $G$ with 
    $|\mathcal{P}|=i$ such that the paths in $\calP$ are pairwise at distance at least $\ell$ in $G$.
\end{lemma}

\begin{proof}
    Let $\calM = (M_x \mid x \in V(F) \cup E(F))$.
    Let $V_\alpha = \{v \in V(F) \mid \deg_F(v) = \alpha\}$ for each $\alpha \in \{0,1,2,3\}$, and let $m$ be the number of components of $F$ with at least two vertices.
    By~\ref{frame:Y}, for every $x \in V_0$, $M_x$ is an $A$-path in $G$.
    Let $F' = F - V_0$, $Z = V_1 \cup V_2$, and $j = \left\lceil\frac{|Z| - m}{2}\right\rceil$.
    By~\Cref{cor:extracting-paths-from-the-frame}, $F'$ contains $j$ pairwise vertex-disjoint $Z$-paths $Q_1,\dots,Q_j$.
    
    For each $\beta \in [j]$, we now define an $A$-path $P_\beta$ in $G$.
    Consider the graph $G_\beta = \bigcup_{x \in V(Q_\beta)\cup E(Q_\beta)} M_x$.
    Since $Q_\beta$ is a path in $F$ and $\calM$ is a model of $F$ in $G$, $G_\beta$ is connected.
    Since $Q_\beta$ is a $Z$-path, it contains at least two vertices of $Z$.
    In particular, by~\ref{frame:Z}, there are two distinct vertices $x$ and $y$ of $V(Q_\beta)$ such that both $M_{x}$ and $M_{y}$ contain a vertex of $A$, 
    say $a_x$ and $a_y$, respectively.
    In particular, $a_x \neq a_y$ as $\calM$ is a model. Let $P_\beta$ be an $a_x$-$a_y$ path in $G_\beta$. Thus, $P_\beta$ is an $A$-path in $G$.
    Since $\calM$ is $\ell$-fat (by~\ref{frame:ell}), we have
    \[\dist_G(a_x,a_y) \geq \dist_G(V(M_x), V(M_y)) \geq \ell.\]
    It follows that $P_\beta$ is $\ell$-coarse.

    Let $\calP = \{M_x \mid x \in V_0\} \cup \{P_\beta \mid \beta \in [j]\}$.
    We have 
    \begin{align*}
        |\calP| = |V_0| + j &= |V_0| + \left\lceil\frac{|V_1| + |V_2| - m}{2}\right\rceil\\
        &\geq \left\lceil\frac{|V_0| + |V_1| + |V_2| - m}{2}\right\rceil = \left\lceil\frac{2i-1}{2}\right\rceil = i &&\text{by~\ref{frame:i}.}
    \end{align*}
    As argued before, $\calP$ is a collection of $A$-paths in $G$.
    Since $\calM$ is $\ell$-fat (by~\ref{frame:ell}), by~\Cref{obs:packing}, paths in $\calP$ are pairwise at distance at least $\ell$ in $G$.
    Finally, if $(F,\calM)$ is coarse, then $V_0 = \emptyset$, and so all $A$-paths in $\calP$ are $\ell$-coarse. 
    This completes the proof.
\end{proof}

We have now everything in hand to wrap up the proofs of~\Cref{theorem:main,theorem:main2}.
Note that these two proofs are almost identical, up to using different parts of~\Cref{lem:larger-frame-or-hitting,lem:large-frame-implies-packing}.
For brevity, we give only one proof, adding \q{(coarse)} in several places.
To get the proof for~\Cref{theorem:main}, one should ignore this addition, and to get the proof of~\Cref{theorem:main2}, one should not.

\begin{proof}[Proof of~\Cref{theorem:main,theorem:main2}]
    For all positive integers $k$ and $d$, we set 
    \[f(k) = 4k-4 \text{ and } g(k,d) = 256^{k}d.\]
    We let $k,d$ be positive integers, $G$ be a graph, and $A$ a subset of the vertices of $G$.
    We moreover let $r = 4 \cdot 16^{2k-2}d$.
    We prove that either
    \begin{enumerate}
        \item $G$ contains a (coarse) $(2k-1, d, r, A)$-frame, or \label{item:i-wrapper}
        \item there exists a set $X$ of the vertices of $G$ with $|X| \leq f(k)$ such that every ($g(k,d)$-coarse) $A$-path in $G$ contains a vertex in $B_G(X, g(k,d))$. \label{item:ii-wrapper}
    \end{enumerate}
    Observe that this will immediately conclude the proof of \cref{theorem:main,theorem:main2}, since by~\Cref{lem:large-frame-implies-packing}, if $G$ contains a (coarse) $(2k-1, d, r, A)$-frame, then it also contains $k$ ($d$-coarse) $A$-paths pairwise at distance at least $d$ in $G$.
    
    Assume that~\ref{item:ii-wrapper} does not hold.
    We prove by induction that for every integer $i$ with $0 \leq i \leq 2k-1$, $G$ contains a (coarse) $(i, 16^{2k-1-i}d,r,A)$-frame.
    In particular, the case $i = 2k-1$ corresponds to~\ref{item:i-wrapper}.

    For the base case where $i = 0$, we choose $F$ to be the null graph and $\calM$ to be the empty collection, and we note that $(F,\calM)$ is a (coarse) $(0,16^{2k-1}d,r,A)$-frame in $G$.
    Next, we let $i$ be an integer with $0 \leq i \leq 2k-2$ and assume by induction hypothesis that $G$ contains a (coarse) $(i,16^{2k-1-i}d,r,A)$-frame.
    Let $\ell = 16^{2k-1-i}d \slash 16 = 16^{2k-1-(i+1)}d$.
    Note that $4\ell = 4 \cdot 16^{2k-2-i}d \leq 4\cdot 16^{2k-2}d = r$.
    Thus, we may apply~\Cref{lem:larger-frame-or-hitting}, obtaining that either $G$ contains a (coarse) $(i+1,16^{2k-1-(i+1)}d,r,A)$-frame and the inductive step is completed, or there exists $X \subset V(G)$ with $|X| \leq 2i$ such that every ($\ell$-coarse) $A$-path in $G$ contains a vertex in $B_G(X, r+8\ell)$.
    Note that $2i \leq 2(2k-2) = f(k)$, $r+8\ell =  4\cdot 16^{2k-2}d + 8 \cdot 16^{2k-2-i}d \leq 12 \cdot 16^{2k-2}d < g(k,d)$, and $\ell < g(k,d)$.
    Therefore, the latter possibility in the outcome of~\Cref{lem:larger-frame-or-hitting} would contradict our assumption that~\ref{item:ii-wrapper} does not hold, concluding the induction step.
\end{proof}

\section{Coarse equivalence between minors and topological minors}
\label{sec:topological-minors}

In this section, we prove~\Cref{lemma:showcase-of-the-tripod-lemma}, whose statement we repeat for convenience.

\topologicalminors*

\begin{proof}
Let $\mathcal{M} = (M_x \mid x \in V(H)\cup E(H))$ be a $7\ell$-fat model of $H$ in $G$. 
Let $uv$ be an edge in $H$.
Note that $M_{uv}$ is a connected graph containing a vertex of $B_G(V(M_u),2\ell)$ and a vertex of $B_G(V(M_v),2\ell)$. 
Let $W_{uv}$ be a $B_G(V(M_u),2\ell)$-$B_G(V(M_v),2\ell)$ path contained in $M_{uv}$. 
Let $x_{uv,u}$ and $x_{uv,v}$ denote respectively the endpoints of $W_{uv}$ in $B_G(V(M_u),2\ell)$ and 
in $B_G(V(M_v),2\ell)$. Note that we assumed $\mathcal M$ to be $7\ell$-fat, so $B_G(V(M_u),2\ell)$ and $B_G(V(M_v),2\ell)$ are disjoint, thus $x_{uv,u}$ and $x_{uv,v}$ are at distance in $G$ exactly $2\ell$ from, respectively, $M_u$ and $M_v$. 

For each vertex $u$ of $G$, we will define a connected subgraph $N_u$ of $G$, and for each edge $e$ incident to $u$ in $H$, we will define a connected subgraph $P_{e,u}$ of $G$ such that denoting $e_{1},\ldots,e_\delta$ (where $\delta =\deg_H(u)\leq 3$) the edges incident to $u$ in $H$, we have
\begin{enumerate'}
\item\label{it:top-minor-incidence} $x_{e_i,u}\in V(P_{e_i,u})$ and $V(N_u)\cap V(P_{e_i,u})\neq\emptyset$ for each $i\in[\delta]$, 
\item\label{it:top-minor-radius} $N_u\ \textrm{has radius at most $\lfloor1.5\ell\rfloor$}$,
\item\label{it:top-minor-branchsets} $V(N_u)\subseteq B_G(V(M_u),2\ell)$,
\item\label{it:top-minor-branchpaths} $V(P_{e_i,u})  \subseteq B_G(x_{e_i,u},\ell-1) \cup B_G(V(M_u),\ell)$ for each $i\in[\delta]$,
\item\label{it:top-minor-fat} $\dist_G(V(P_{e_i,u}),V(P_{e_j,u}))\geq\ell$ for all distinct $i,j\in[\delta]$.
\end{enumerate'}
In particular,~\ref{it:top-minor-branchpaths} implies 
\begin{enumerate'}
    \setcounter{enumi}{5}
    \item\label{it:top-minor-branchpaths+} $V(P_{e_i,u})  \subseteq B_G(V(M_u),3\ell)$ for each $i\in [\delta]$.
\end{enumerate'}

Let $u$ be a vertex of $H$.
The definition of $N_{u}$ and of the subgraphs $P_{e,u}$ for each $e \in E(H)$ incident to $u$ in $H$ depends on the degree of $u$. 
Assume first that $u$ has degree $3$ in $H$ and let $e_1$, $e_2$, $e_3$ denote the three edges incident to $u$ in $H$. 
We apply the \nameref{lemma:tripod} choosing 
$\set{v_i}_{i\in[3]}=\set{x_{e_i,u}}_{i\in[3]}$, $Q=M_u$ and $d=2\ell$. 
Note that assumptions~\eqref{lem:tripod-asumpt-dist-xi-Q-large} and~\eqref{lem:tripod-asumpt-dist-xi-Q-small} hold as $\dist_G(x_{e_i,u},V(M_u))=2\ell$ 
for each $i\in[3]$, and assumption~\eqref{lem:tripod-asumpt-dist-xi-xj} holds as for all distinct $i,j \in [3]$,
$$\dist_G(x_{e_i,u},x_{e_j,u}) \geq \dist_G(V(M_{e_i}), V(M_{e_j})) \geq 7\ell\geq 2d.$$
We therefore obtain a connected subgraph $N_{u}=Z$ of $G$, together with three connected subgraphs $\set{P_{e_i,u}}_{i\in[3]}$ of $G$ that satisfy conditions \ref{it:top-minor-incidence} to \ref{it:top-minor-branchpaths+}.

Next, assume that $u$ has degree $2$ in $H$ and let $e_1$, $e_2$ denote the two edges incident to $u$ in $H$. 

For each $i \in [2]$, let $Q_i$ be a shortest $x_{e_i,u}$-$V(M_u)$ path in $G$, and let $y_i$ be the endpoint of $Q_i$ in $V(M_u)$. 
We define  
\[N_u = Q_1, \ \ P_{e_1,u} = (x_{e_1,u}, \emptyset), \ \text{ and} \ \ P_{e_2,u} = M_u \cup Q_2.\]
As $x_{e_2,u} \in V(P_{e_2,u})$, $y_1 \in V(N_u) \cap V(P_{e_2,u})$ and $x_{e_1,u} \in V(P_{e_1,u}) \cap V(N_u)$,~\ref{it:top-minor-incidence} holds for $u$.
As $N_u=Q_1$ is a shortest path of length $2\ell$, it has radius $\ell\leq \lfloor1.5\ell\rfloor$ in $G$, showing~\ref{it:top-minor-radius}. 
Condition~\ref{it:top-minor-branchsets} follows from the fact that $Q_1$ is a shortest $x_{e_1,u}$-$M_u$ path, and that $x_{e_1,u}$ is at distance exactly $2\ell$ from $M_u$, hence $N_u=Q_1$ must be included in $B_G(M_u,2\ell)$. Since $Q_2$ has length $2\ell$ and $x_{e_2,u},y_2$ are its endpoints, we have $V(Q_2)\subseteq B_G(x_{e_2,u},\ell-1)\cup B_G(y_2,\ell)\subseteq B_G(x_{e_2,u},\ell-1)\cup B_G(V(M_u),\ell)$, and thus $u$ satisfies \ref{it:top-minor-branchpaths} and \ref{it:top-minor-branchpaths+}.
Finally, \ref{it:top-minor-fat} also holds, because
\begin{align*}
    \dist_G(V(P_{e_i,u}),V(P_{e_j,u})) &= \dist_G(x_{e_1,u},V(M_u)\cup V(Q_2))\\
    &\geq \min\{\dist_G(x_{e_1,u},V(M_u)), \dist_G(x_{e_1,u},V(Q_2))\}\\
    &\geq \min\{2\ell, \dist_G(x_{e_1,u},x_{e_2,u}) - 2\ell \} && \textrm{as } \length(Q_2) = 2\ell\\
    &\geq \min\{2\ell, 7\ell - 2\ell\} \ge \ell && \textrm{as $\calM$ is $7\ell$-fat}.
\end{align*}

Finally, assume that $u$ has degree $1$ in $H$ and let $e$ be the edge incident to $u$ in $H$.
Let $P_{e,u}$ be a shortest $x_{e,u}$-$V(M_u)$ path in $G$ and denote $u'$ the endpoint of $P_{e,u}$ in $V(M_u)$. Note that $P_{e,u}$ has length $2\ell$.  
We define $N_u$ as the one vertex graph containing $u'$. Conditions \ref{it:top-minor-incidence} to~\ref{it:top-minor-branchpaths+} immediately follow.

This completes the construction of subgraphs $N_u$ for each $u \in V(G)$ and $P_{e,u}$ for each $e\in E(H)$ incident to $u$ satisfying conditions \ref{it:top-minor-incidence} to~\ref{it:top-minor-branchpaths+}.
For each edge $e=uv$ in $H$, we define
\[
N_{e} = P_{e,u} \cup W_{e} \cup P_{e,v},
\]
and we set $\mathcal N=(N_x \mid x\in V(H)\cup E(H))$.

For each $u \in V(H)$, $N_u$ is a connected subgraph of $G$, and by~\ref{it:top-minor-incidence}, for each $e \in E(H)$, $N_e$ is a connected subgraph of $G$.
Moreover, for each vertex $u\in V(H)$ and each edge $e\in E(H)$ incident to $u$ in $H$, $x_{e,u}$ witnesses that $N_u\cap N_e\neq \emptyset$ (by~\ref{it:top-minor-incidence} again).
Let $x$ and $y$ be distinct elements of $V(H) \cup E(H)$ so that $\set{x,y}$ is not equal to $\set{v,e}$ where $v \in V(H)$, $e \in E(H)$, and $v$ is incident to $e$ in $H$. 
We consider all possible choices of $x$ and $y$ up to swapping them.
In each case, we prove that $\dist_G(V(N_x),V(N_{y}))\geq\ell$.
This will show that $\calN$ is a model of $H$ in $G$ (as $\ell > 0$) and that $\calN$ is an $\ell$-fat model.

If $x,y\in V(H)$, then
\begin{align*}
\dist_G(V(N_x),V(N_y)) 
&\geq \dist_G(B_G(V(M_x),2\ell),B_G(V(M_y),2\ell))&&\textrm{by~\ref{it:top-minor-branchsets}} \\
&\geq \dist_G(V(M_x),V(M_y))-2\cdot 2\ell\\
&\geq 7\ell-4\ell \geq \ell&&\textrm{as $\mathcal{M}$ is $7\ell$-fat}.
\end{align*}

If $\{x,y\} = \{u,vw\}$ where $u$, $v$ and $w$ are distinct vertices of $H$ and $vw \in E(H)$, then 
\begin{align*}
\dist_G(V(N_u),V(N_{vw})) 
&\geq \min\{\dist_G(B_G(V(M_u),2\ell),V(P_{vw,v})),&&\textrm{by~\ref{it:top-minor-branchsets} and}\\
& \hphantom{\geq \min \{} \dist_G(B_G(V(M_u),2\ell), V(W_{vw})), &&\textrm{definition of $N_{vw}$}\\
& \hphantom{\geq \min \{} \dist_G(B_G(V(M_u),2\ell), V(P_{vw,w}))\}.
\end{align*}
We have $V(W_{vw}) \subseteq V(M_{vw})$, so $\dist_G(B_G(V(M_u),2\ell), V(W_{vw})) \ge 7\ell-2\ell \ge \ell $ as $\calM$ is $7\ell$-fat. We have 
\begin{align*}
\dist_G(B_G(V(M_u),2\ell),V(P_{vw,v}))&\geq \dist_G(B_G(V(M_u),2\ell),B_G(V(M_v),3\ell))&&\textrm{by~\ref{it:top-minor-branchpaths+}}\\
& \geq 7\ell-5\ell \ge \ell &&\textrm{as $\calM$ is $7\ell$-fat.}
\end{align*}
Symmetrically, we also have $\dist_G(B_G(V(M_u),2\ell),V(P_{vw,w}))\ge \ell$, so $\dist_G(V(N_u),V(N_{vw}))\ge \ell$.

If $\{x,y\} = \{e_1,e_2\}$ where $e_1$ and $e_2$ are distinct non-incident edges of $H$, say $e_1 = u_1v_1$ and $e_2 = u_2v_2$, then for each $i\in [2]$, $V(N_{e_i}) \subseteq B_G(V(M_{u_i}),3\ell) \cup V(M_{e_i}) \cup B_G(V(M_{v_i}),3\ell)$ by~\ref{it:top-minor-branchpaths+}. So
\begin{align*}
\dist_G(V(N_{e_1}),V(N_{e_2})) 
&\geq \dist_G(V(M_{u_1}) \cup V(M_{e_1}) \cup V(M_{v_1}), \\
&\hphantom{\geq \dist_G(}V(M_{u_2}) \cup V(M_{e_2}) \cup V(M_{v_2}))-2\cdot3\ell\\
&\geq 7\ell -6\ell \ge \ell &&\textrm{as $\calM$ is $7\ell$-fat.}
\end{align*}

If $\{x,y\} = \{uv,uw\}$ where $u$, $v$ and $w$ are distinct vertices of $H$, and $uv$ and $vw$ are edges of $H$, then recall that $N_{uv} = P_{uv,u} \cup (W_{uv} \cup P_{uv,v})$, and likewise for $N_{uw}$, so we have
\begin{align*}
\dist_G(V(N_{uv}),V(N_{uw})) 
&\geq \min\{\dist_G(V(P_{uv,u}),V(P_{uw,u})),\\
& \hphantom{\geq \min\{}\dist_G(V(P_{uv,u}),V(W_{uw} \cup P_{uw,w})), \\
& \hphantom{\geq \min\{}\dist_G(V(W_{uv}\cup P_{uv,v}),V(P_{uw,u})), \\
& \hphantom{\geq \min\{}\dist_G(V(W_{uv} \cup P_{uv,v}),V(W_{uw} \cup P_{uw,w}))\} \\
&\geq \min\{\ell, &&\textrm{by~\ref{it:top-minor-fat}}\\
& \hphantom{\geq \min\{}\dist_G(V(P_{uv,u}),V(W_{uw}) \cup B_G(V(M_w),3\ell)), && \textrm{by~\ref{it:top-minor-branchpaths+}} \\
& \hphantom{\geq \min\{}\dist_G(V(W_{uv}) \cup B_G(V(M_v),3\ell),V(P_{uw,u})),  && \textrm{by~\ref{it:top-minor-branchpaths+}}\\
& \hphantom{\geq \min\{}\dist_G(V(W_{uv}) \cup B_G(V(M_v),3\ell),&& \textrm{by~\ref{it:top-minor-branchpaths+}}\\
& \hphantom{\geq \min\{\dist_G}V(W_{uw}) \cup B_G(V(M_w),3\ell))\} && \textrm{by~\ref{it:top-minor-branchpaths+}.}\\
\end{align*}
Since $W_{uv} \subset M_{uv}$ and $W_{uw} \subset M_{uw}$ and $\calM$ is $7\ell$-fat, we have
\[\dist_G(V(W_{uv}) \cup B_G(V(M_v),3\ell),V(W_{uw}) \cup B_G(V(M_w),3\ell))\ge 7\ell-2\cdot3\ell\ge \ell.\]
Next, we have
 \begin{align*}
 \dist_G(V(P_{uv,u}),V(W_{uw}) &\cup B_G(V(M_w),3\ell))\\
 &\ge \min\{\dist_G(V(P_{uv,u}),V(W_{uw})),\\ 
 & \hphantom{\geq \min\{}\dist_G(V(P_{uv,v}), V(M_w))-3\ell\} \\
  &\ge \min\{\dist_G(B(V(M_u),3\ell),V(M_{uw}))), && \textrm{by~\ref{it:top-minor-branchpaths+}} \\ 
 & \hphantom{\geq \min\{}\dist_G(B_G(V(M_v),3\ell), V(M_w))-3\ell\} && \textrm{by~\ref{it:top-minor-branchpaths+}} \\
  &\ge \min\{\dist_G(V(M_u),V(M_{uw})) - 3\ell, \\
 & \hphantom{\geq \min\{}\dist_G(V(M_v), V(M_w))-6\ell\} \\
 &\geq \min \{7\ell-3\ell, 7\ell-6\ell\} = \ell && \textrm{as $\calM$ is $7\ell$-fat.}
 \end{align*}
Symmetrically, $\dist_G(V(W_{uv}) \cup B_G(M_v,3\ell),V(P_{uw,u})) \geq \ell$.
Thus, $\dist_G(V(N_{uv}),V(N_{uw}))  \geq \ell$.

This concludes the proof that $\calN$ is an $\ell$-fat model of $H$, and so,~\ref{it:thm-top-minor-fat} holds.
By~\ref{it:top-minor-radius},~\ref{it:thm-top-minor-radius} holds, and altogether, this completes the proof.
\end{proof}

\bibliographystyle{abbrvnat}
\bibliography{bibliography}
\end{document}